\documentclass[12pt]{amsart}
\usepackage{graphicx}
\usepackage[top=1in,bottom=1in,left=1.25in,right=1.25in]{geometry} 
\usepackage{amsmath,amssymb,amsfonts}
\usepackage{amsthm}
\usepackage{color}
\usepackage{hyperref}
\newtheorem{theo}{Theorem}[section]
\newtheorem{prop}[theo]{Proposition}
\newtheorem{lemma}[theo]{Lemma}
\newtheorem{defn}[theo]{Definition}

\newtheorem{cor}[theo]{Corollary}
\newcommand{\cal}{\mathcal}

\begin{document}
\title{On the knot Floer homology of a class of satellite knots}
\author{Yuanyuan Bao}
\address{
Department of Mathematics,
Tokyo Institute of Technology,
Oh-okayama, Meguro, Tokyo 152-8551, Japan
}
\email{bao.y.aa@m.titech.ac.jp}
\date{}
\begin{abstract}
Knot Floer homology is an invariant for knots in the three-sphere for which the Euler characteristic is the Alexander-Conway polynomial of the knot. The aim of this paper is to study this homology for a class of satellite knots, so as to see how a certain relation between the Alexander-Conway polynomials of the satellite, companion and pattern is generalized on the level of the knot Floer homology. We also use our observations to study a classical geometric invariant, the Seifert genus, of our satellite knots.
\end{abstract}
\keywords{knot Floer homology, satellite knot, Seifert genus}
\thanks{The author is supported by scholarship from
the Ministry of Education, Culture, Sports, Science and Technology of Japan.}
\maketitle

{
\section{Introduction}
Given a link $L$ in the three-sphere $S^{3}$, the knot Floer homology of $L$ \cite{MR2065507, Rasmussenthesis} is denoted 
$\bigoplus_{i\in {\mathbb Z}}\widehat{\operatorname {HFK}}(S^{3}, L, i)$ with $i$ being the Alexander grading. It generalizes the Alexander-Conway polynomial in the following sense:
\begin{theo}[\cite{MR2065507}, \cite{Rasmussenthesis}]
\label{euler}
Given an $l$-component link $L\subset S^{3}$, we see
$$\sum_{i}\chi(\widehat{\operatorname {HFK}}(S^{3}, L, i))\cdot T^{i}=(T^{1/2}-T^{-1/2})^{l-1}\Delta_{L}(T).$$
Here $\chi(\widehat{\operatorname {HFK}}(S^{3}, L, i))$ denotes the Euler characteristic, and $\Delta_{L}(T)$ denotes the normalized Alexander-Conway polynomial of $L$.
\end{theo}

Therefore, it is natural to ask whether a given property of the Alexander-Conway polynomial has a generalization in the context of  knot Floer homology. This homology has proven to be a useful tool for studying some geometric properties of knots, such as the sliceness \cite{MR2026543} and the Seifert genus \cite{MR2023281} etc. Therefore, any new observaton would possibly give us new hints and insights into understanding these properties of knots.

When considering satellite knots, we know the following Equation \eqref{alex}. Given a knot $K\subset S^{3}$ and a non-trivially properly embedded simple closed curve $P\subset \textbf{S}$, where $\textbf{S}$ is a solid torus, we let $K_{t}^{P}$ denote the $t$-twisted satellite knot for $t\in {\mathbb Z}$. Here $K$ and $P$ work as the companion and pattern of $K_{t}^{P}$, respectively. The sign of a full-twist is defined in Figure \ref{fig:f18}. The Alexander-Conway polynomial satisfies the following equation (see \cite{MR1472978}):
\begin{equation}
\label{alex}
\Delta_{K_{t}^{P}}(T)=\Delta_{K}(T^{\xi})\cdot\Delta_{O_{t}^{P}}(T).
\end{equation}
Here $O$ denotes the unknot, and $\xi$ is an integer such that $\xi$ times a generator of $H_{1}(\textbf{S}, {\mathbb Z})$ equals $[P]\in H_{1}(\textbf{S}, {\mathbb Z})$. 

\begin{figure}[h]
 \setlength{\abovecaptionskip}{-0.1cm}
	\centering
		\includegraphics[width=0.6\textwidth]{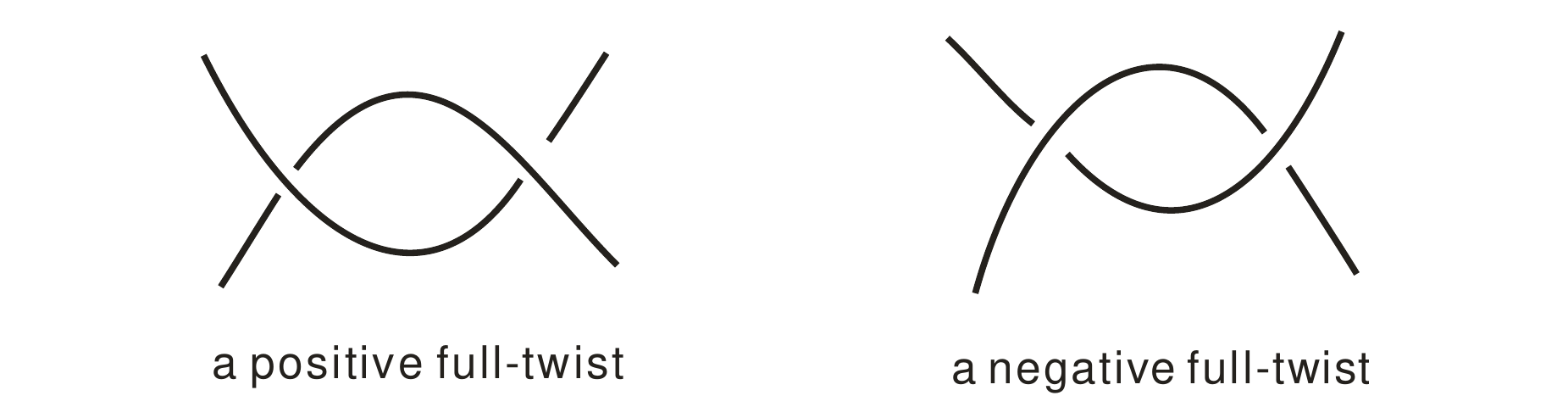}
	\caption{The sign convention for a full-twist.}
	\label{fig:f18}	
\end{figure}

In this paper, we study the knot Floer homology of $(S^{3}, K_{t}^{P})$ for a class of given patterns. One of our purposes is to see how Equation (\ref{alex}) is developed in knot Floer homology. 

This subject has been studied by Hedden in two cases. The $(p, pn\pm 1)$ cabled knots were studied in \cite{Heddenthesis}, where he showed the homology of the cabled knot of a knot $K$ depends only on the filtered chain homotopy type of $\widehat{\operatorname {CFK}}(S^{3}, K)$ when $|n|\gg 0$. In \cite{MR2372849} Hedden studied the knot Floer homology of Whitehead doubles comprehensively. As one important application, he showed a way to find infinitely many topologically slice but not smoothly slice knots. For other research related to this topic, the reader is referred to \cite{Ordingthesis} and \cite{MR2171814}.

Before stating our results, we review some information about the knot Floer homology theory. 
Ozsv{\'a}th and Szab{\'o} \cite{MR2113019} defined a homology theory for 
oriented closed 3-manifolds, known as the Heegaard Floer homology or Ozsv{\'a}th-Szab{\'o} homology. In this paper, we focus on the hat version of this homology. The chain complex associated with a 3-manifold $M$ is denoted $\widehat{\operatorname {CF}}(M)$, and its homology $\widehat{\operatorname {HF}}(M)$ is a topological invariant. Given a null-homologous embedded curve $K$ in the 3-manifold $M$, Ozsv$\rm{\acute{a}}$th and Szab$\rm{\acute{o}}$ \cite{MR2065507}, and Rasmussen \cite{Rasmussenthesis} independently noticed that $K$ induces a filtration to the complex $\widehat{\operatorname {CF}}(M)$, and they proved that the filtered chain homotopy type with respect to the new filtration is a knot invariant. When $M=S^{3}$, let $F(K,m)$ be the subcomplex of $\widehat{\operatorname {CF}}(S^{3})$ of filtration $m$ for $m \in {\mathbb Z}$ and let $\widehat{\operatorname {HF}}(F(K, m))$ denote the homology of the subcomplex $F(K,m)$. Then there exist inclusive relations:
$$0\subset\cdots\subset F(K, m)\subset F(K, m+1)\subset\cdots\subset \widehat{\operatorname {CF}}(S^{3}).$$
Let ${\widehat{\operatorname {CF}}(S^{3})}/{F(K, m)}$ denote the quotient complex of $F(K, m)$ in $\widehat{\operatorname {CF}}(S^{3})$. 

The associated graded chain complex is $\widehat{\operatorname {CFK}}(S^{3}, K, m)={F(K, m)}/{F(K, m-1)}$ and its homology is denoted $\widehat{\operatorname {HFK}}(S^{3}, K, m)$. The homological grading of the homology group is renamed the Maslov grading, and the new grading induced from the filtration is called the Alexander grading. 
The Ozsv$\rm{\acute{a}}$th-Szab$\rm{\acute{o}}$ $\tau$ invariant is defined as:
$$\tau(K)={\rm min}\{m\in {\mathbb Z}| i_{*}: \widehat{\operatorname {HF}}(F(K, m))\longrightarrow \widehat{\operatorname {HF}}(S^{3}) \text{ is non-trivial}\}.$$ 

Now we define the patterns to be used here. Consider the tangle $(B^{3}, S_{r})$ defined in Figure \ref{fig:f1} for any $r\in {\mathbb Z}_{\geq 0}$. Inside a rectangle marked by an odd (even, respectively) number is a negative (positive, respectively) full-twist. The four endpoints of the tangle are $A$, $B$, $C$, and $D$. Connecting $A$ to $D$, and $B$ to $C$ gets the 2-bridge knot $C(\underbrace{-2,-2,\dots,-2}_{2r})$ in Conway's normal form, which is denoted $C(2r)$ here for short, while connecting $A$ to $B$, and $C$ to $D$ gives rise to the 2-bridge link $C(\underbrace{-2,-2,\dots,-2}_{2r+1})$, and it is denoted $C(2r+1)$. The patterns used in this paper are shown in the middle of Figure \ref{fig:f1}. They are embeddings of $\{C(2r)\}_{r\geq 0}$ into the solid torus $\textbf{S}$, and we denote them by $\{P_{r}\}_{r\geq 0}$. When $r=0$, the corresponding satellite knots are the Whitehead doubles. 

\begin{figure}[h]
	\centering
		\includegraphics[width=1.0\textwidth]{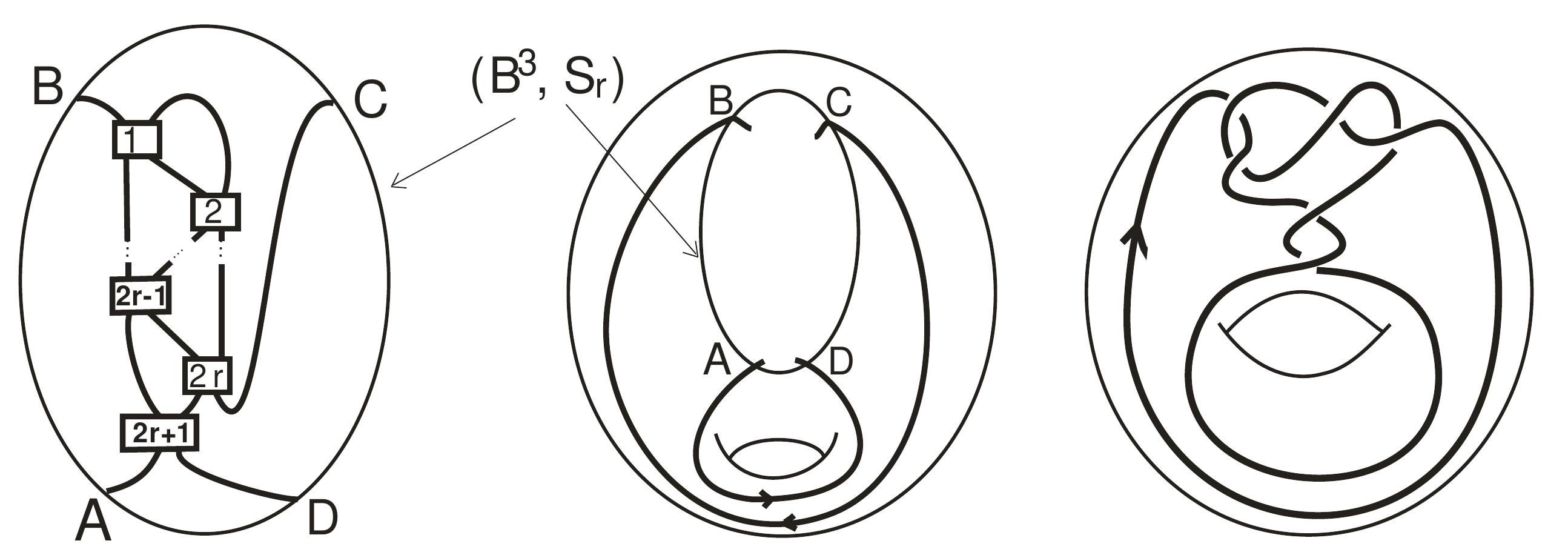}
	\caption{The left-hand figure is the tangle $(B^{3}, S_{r})$. The middle figure illustrates the patterns in the solid torus $\textbf{S}$. The right-hand figure shows the pattern in the case $r=1$ as an example.}
	\label{fig:f1}
\end{figure}

The research in this paper is motivated by Hedden's research on Whitehead doubles in \cite{MR2372849}. We found that many of his ideas can be used to study a huge class of satellite knots. Many interesting applications are given in \cite{MR2372849}, so we hope to get more applications, especially those to some geometric invariants of knots, by studying more satellite knots. We choose $\{P_{r}\}_{r\geq 0}$ as the patterns since the satellite knots obtained are certain extensions of the Whitehead doubles.
 
Our main result is as follows.

\begin{theo}
\label{main1}
Let $K\subset S^{3}$ be a knot with Seifert genus $g(K)=g$. Then the top Alexander grading of $\widehat{\operatorname {HFK}}(S^{3}, K_{t}^{P_{r}})$ is $r+1$ when $t\neq 0$ for $r\in {\mathbb Z}_{\geq 0}$. At this grading, there exists an integer $N$ so that for $t>N>0$, the following hold.
$$\widehat{\operatorname {HFK}}_{\ast}(S^{3}, K_{t}^{P_{r}}, r+1)\cong {\mathbb Z}^{t-2g-2}_{(r+1)}\bigoplus^{g}_{m=-g}\left[\widehat{\operatorname {HF}}_{\ast-r-1}(F(K, m))\right]^{2},$$
$$\widehat{\operatorname {HFK}}_{\ast}(S^{3}, K_{-t}^{P_{r}}, r+1)\cong {\rm Tor}^{\pm} \bigoplus {\mathbb Z}^{2\tau(K)-2g-2}_{(r+1)}\bigoplus {\mathbb Z}^{2\tau(K)+t}_{(r)}\bigoplus^{g}_{m=-g}\left[\widehat{\operatorname {HF}}_{\ast-r-1}(F(K, m))\right]^{2},$$
where {\rm Tor} denotes a finite Abelian group, and the subindices of ${\mathbb Z}$ represent the Maslov gradings.
\end{theo}

The convention for equations in Theorem \ref{main1} is, if the power of a summand on the right side of an equation is negative, we simply move this summand to the left-hand side of the equation and convert the power into its opposite value, just as we usually do for multiplication of numbers.

From Equation (\ref{alex}), the Alexander-Conway polynomial of $K_{t}^{P_{r}}$ satisfies
\begin{equation}
\label{shape}
\Delta_{K_{t}^{P_{r}}}(T)=\Delta_{O_{t}^{P_{r}}}(T),
\end{equation}
since $\xi=0$ for all the patterns in $\{P_{r}\}_{r\geq 0}$. In other words, the value of the Alexander polynomial $\Delta_{K_{t}^{P_{r}}}(T)$ does not depend on the choice of the companion knot $K$. On the other hand, we see from Theorem \ref{main1} that the knot Floer homology $\widehat{\operatorname {HFK}}(S^{3}, K_{t}^{P_{r}})$ at the grading $r+1$ with large twist
is determined by the filtered chain homotopy type of $\widehat{\operatorname {CFK}}(S^{3}, K)$. This observation implies again that knot Floer homology is much more powerful then the Alexander-Conway polynomial in revealing properties of knots. A proof of Theorem \ref{main1} is given in Section 3. 

From Theorem \ref{main1}, we see that the top Alexander grading of $\widehat{\operatorname {HFK}}(S^{3}, K_{t}^{P_{r}})$ when $t$ is not equal to zero is $r+1$, which only depends on the pattern $P_{r}$. This fact is proved here directly from the definition of knot Floer homology. It is in line with Equation (\ref{shape}) from the viewpoint of Theorem \ref{euler}. Since the Seifert genus of a knot can be detected from its knot Floer homology (see Theorem \ref{seifert}), the Seifert genus of $K_{t}^{P_{r}}$ for $t\neq 0$ is $r+1$. As an application and also a complement to Theorem \ref{main1}, we show that the top Alexander grading of $\widehat{\operatorname {HFK}}(S^{3}, K_{0}^{P_{r}})$ for $r\in {\mathbb Z}_{\geq 0}$ is $r+1$ as well for non-trivial knot $K\subset S^{3}$, which implies the Seifert genus $g(K_{0}^{P_{r}})$ is $r+1$. We state this result as Corollary \ref{main2}. Note that $g(K_{0}^{P_{r}})$ cannot be detected from the Alexander-Conway polynomial $\Delta_{K_{0}^{P_{r}}}(T)$ since ${\rm deg}(\Delta_{K_{0}^{P_{r}}}(T))=2r$.

\begin{cor}
\label{main2}
For any non-trivial knot $K\subset S^{3}$, the Seifert genus $g(K_{0}^{P_{r}})$ of the satellite knot $K_{0}^{P_{r}}$ is $r+1$.
\end{cor}
}

The paper is organized as follows. In Section 2, we construct Heegaard diagrams for the satellite knots $K_{t}^{P_{r}}$ for all $r\in {\mathbb Z}_{\geq 0}$ and $t\neq 0$. In Section 3, we study the knot Floer homology of $K_{t}^{P_{r}}$. We start by studying the Alexander gradings of some generators of $\widehat{\operatorname {CFK}}(S^{3}, K_{t}^{P_{r}})$. After that, we focus on the top Alexander grading level, and prove Theorem~\ref{main1}. In Section 4 we prove Corollary \ref{main2}.

\section{Heegaard diagrams for the satellite knots $\{K_{t}^{P_{r}}\}_{r\geq 0}$}

\subsection{A brief review of knot Floer homology in $S^{3}$}

\begin{defn}
\label{heegaard}
\rm
For a null-homologous simple closed curve $K$ in an oriented closed 3-manifold $M$, {\it a doubly-pointed Heegaard diagram} ({\it simply a Heegaard diagram}) for $(M, K)$ is a collection of data $(\Sigma_{g}, \boldsymbol{\alpha}=\{\alpha_{1}, \alpha_{2}, \cdots, \alpha_{g}\}, \boldsymbol{\beta}=\{\beta_{1}, \beta_{2}, \cdots, \beta_{g}\},$ $ z, w)$, which satisfies the following conditions.
\begin{enumerate}
	\item The surface $\Sigma_{g}$ splits $M$ into two handlebodies $U_{1}$ and $U_{2}$. It is called the Heegaard surface, and its genus $g$ is called the Heegaard genus.  
	\item The sets $\boldsymbol{\alpha}$ and $\boldsymbol{\beta}$ are two collections of pairwise disjoint essential curves in $\Sigma_{g}$ such that $U_{1}$ and $U_{2}$ are obtained by attaching 2-handles to $\Sigma_{g}$ along the $\boldsymbol{\alpha}$ and $\boldsymbol{\beta}$ curves, respectively.
	\item There exist two arcs $a$ and $b$ in $\Sigma_{g}$ with common endpoints $z$ and $w$ such that $a\cap(\bigcup_{i=1}^{g}\alpha_{i})=\emptyset$ and $b\cap(\bigcup_{i=1}^{g}\beta_{i})=\emptyset$. The knot $K$ is isotopic to the union $a\cup b$ after pushing the arcs $a$ and $b$ properly into the handlebodies $U_{1}$ and $U_{2}$ respectively. 
	\end{enumerate}
\end{defn}

Given a knot $K\subset S^{3}$, let $(\Sigma_{g}, \boldsymbol{\alpha}, \boldsymbol{\beta}, z, w)$ be a Heegaard diagram for the knot $(S^{3}, K)$. Define a set:
$$S=\{x=(x_{1}, \cdots, x_{g})|x_{i}\in\alpha_{i}\cap \beta_{\sigma(i)}, \,\,1\leq i \leq g,\,\, \sigma\in S_{g}\}.$$ The chain complex associated with the Heegaard diagram $(\Sigma_{g}, \boldsymbol{\alpha}, \boldsymbol{\beta}, z, w)$ is freely generated by $S$.
Suppose $D_{1},\ldots, D_{m}$ denote the closures of the components of $\Sigma_{g}-\boldsymbol{\alpha}-\boldsymbol{\beta}$. A domain in $\Sigma_{g}$ is a two-chain ${\cal D}=\sum_{i=1}^{m}a_{i}\cdot D_{i}$ for $a_{i} \in {\mathbb Z}$. For two domains ${\cal D}_{1}=\sum_{i=1}^{m}a_{i}\cdot D_{i}$ and ${\cal D}_{2}=\sum_{i=1}^{m}b_{i}\cdot D_{i}$, the sum used in this paper is defined as follows:
$${\cal D}_{1}+{\cal D}_{2}:=\sum_{i=1}^{m}(a_{i}+b_{i})\cdot D_{i}.$$
The local multiplicity $n_{d}({\cal D})$ of a point $d \in {\rm int}(D_{i})\subset\Sigma_{g}$ at ${\cal D}$ is defined by $n_{d}({\cal D}):=a_{i}$. Given two generators $x$ and $y$ in $S$, a domain ${\cal D}$ is said to be connecting $x$ to $y$ if $\partial{\cal D}$ connects $x$ to $y$ along the $\boldsymbol{\alpha}$ curves and connects $y$ to $x$ along the $\boldsymbol{\beta}$ curves. 

Given two generators $x$ and $y$ and a domain ${\cal D}$ connecting $x$ to $y$, the following equations hold:
\begin{equation*}
\begin{split}
A(x)-A(y)=n_{z}({\cal D})-n_{w}({\cal D}),\\
{\rm gr}(x)-{\rm gr}(y)=\mu({\cal D})-2n_{w}({\cal D}), 
\end{split}
\end{equation*}
where $A$ and ${\rm gr}$ represent the Alexander and the Maslov gradings respectively, and $\mu({\cal D})$ is the Maslov index of $\mu({\cal D})$.

We review a combinatorial formula for Maslov index from Lipshitz. Given a domain ${\cal D}=\sum_{i=1}^{m}a_{i}D_{i}$, the point measure $n_{x}({\cal D})$ of ${\cal D}$ at a generator $x=(x_{1}, x_{2}, \cdots, x_{g})$ is defined as 
\begin{equation}
\label{point}
n_{x}({\cal D})=\sum_{i=1}^{g}n_{x_{i}}({\cal D}),
\end{equation}
where $n_{x_{i}}({\cal D})$ is defined to be the average of the coefficients of ${\cal D}$ at the four regions divided by the corresponding $\boldsymbol{\alpha}$ curve and $\boldsymbol{\beta}$ curve around $x_{i}$. Here 
$e({\cal D})=\sum_{i=1}^{m}a_{i}e(D_{i})$,
and if a region $D_{i}$ is a $2n$-gon, then 
\begin{equation}
\label{polygon}
e(D_{i})=1-{n}/{2}.
\end{equation}  

\begin{theo}[\cite{MR2240908}]
\label{lip}
Given two generators $x$ and $y$, let ${\cal D}$ be a domain connecting $x$ to $y$. Then we have
$$\mu({\cal D})=e({\cal D})+n_{x}({\cal D})+n_{y}({\cal D}).$$
\end{theo}

The formula can be simply applied in our calculations since the domains we will consider all consist of polygons of even number of edges.

It is known that the knot Floer homology detects the Seifert genus:
\begin{theo}[\cite{MR2023281}]
\label{seifert}
For any knot $K\subset S^{3}$, the knot Floer homology of $K$ detects the Seifert genus of $K$. Namely
$$g(K)={\rm max}\left\{i\in {\mathbb Z} \left|\bigoplus_{j}\widehat{\operatorname {HFK}}_{j}(S^{3}, K, i)\neq 0\right.\right\}.$$
\end{theo}

There are exact sequences associated with the knot Floer homologies of links in $S^{3}$ (refer to \cite{MR2065507}), which can be regarded as extensions of the skein relation of the Alexander-Conway polynomial. We recall the one to be used in Section 4. Let $L_{+}\subset S^{3}$ be a link, and $p$ be a positive crossing of some projection of $L_{+}$. There are two associated links, $L_{0}$ and $L_{-}$, which agree with $L_{+}$ except at the crossing $p$ (see Figure \ref{fig:f5}). If the two strands projecting to $p$ belong to the same component of $L_{+}$, the skein exact sequence reads:
\begin{equation}
\cdots\longrightarrow \widehat{\operatorname {HFK}}(L_{-})\longrightarrow \widehat{\operatorname {HFK}}(L_{0})\longrightarrow \widehat{\operatorname {HFK}}(L_{+})\longrightarrow\cdots,
\end{equation}
where all the maps above respect the splitting of $\widehat{\operatorname {HFK}}(L)$ under the Alexander grading, and the maps to and from $\widehat{\operatorname {HFK}}(L_{0})$ drop the Maslov grading by ${1}/{2}$. The map from $\widehat{\operatorname {HFK}}(L_{+})$ to $\widehat{\operatorname {HFK}}(L_{-})$ does not necessarily respect the Maslov grading, but it can be expressed as a sum of homogeneous maps, none of which increases the Maslov grading.

\begin{figure}
\setlength{\abovecaptionskip}{-0.4cm}
	\centering
		\includegraphics[width=0.8\textwidth]{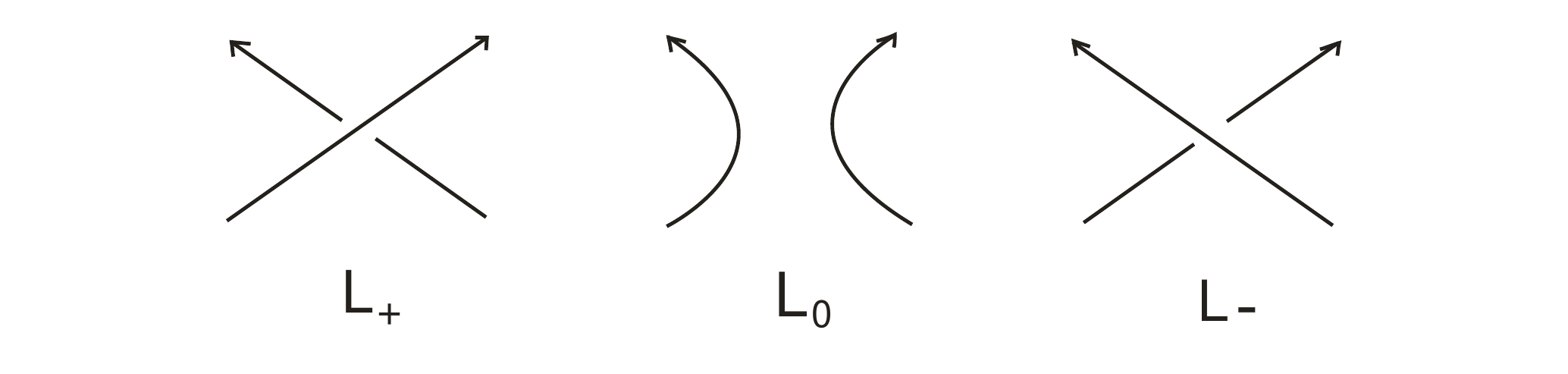}
	\caption{Local diagrams of $L_{+}$, $L_{0}$ and $L_{-}$ around the point $p$.}	
	\label{fig:f5}
\end{figure}

\subsection{Heegaard diagrams for the satellite knots $\{K_{t}^{P_{r}}\}_{r\geq 0}$}
In this subsection, we introduce Heegaard diagrams for the satellite knots $(S^{3}, K_{t}^{P_{r}})$ for $t\neq 0$. We remark that the construction works for general satellite knots. The idea is included in Section 2 of \cite{MR2372849} and in \cite{MR2171814}.

Suppose $K\subset S^{3}$ is an oriented knot and $P\subset \textbf{S}$ is a non-trivially embedded simple closed curve in the standard solid torus $\textbf{S}\subset S^{3}$. Recall that $K_{t}^{P}$ denotes the $t$-twisted satellite knot for which the companion is $K$ and the pattern is $P$. Then there exists a decomposition of $S^{3}-{\rm Int}N(K_{t}^{P})$ along the torus $\partial \textbf{S}$. Precisely it is $$S^{3}-{\rm Int}N(K_{t}^{P})=(S^{3}-{\rm Int}N(K))\bigcup_{\partial N(K)=\partial \textbf{S}} (\textbf{S}-{\rm Int}N(P)).$$ Translating into the language of Heegaard diagrams, one can create a Heegaard diagram for the knot $(S^{3}, K_{t}^{P})$ by combining a Heegaard diagram for the knot $(S^{3},K)$ with a Heegaard diagram for the knot $(S^{3},P)$. Using this idea, we make a doubly-pointed Heegaard diagram for the satellite knot $(S^{3}, K_{t}^{P_{r}})$ for each $r\geq 0$. 

We first construct a Heegaard diagram for $(S^{3}, P_{r})$ (see Figure \ref{fig:f2}). When $r=0$, that is, the case of the Whitehead double, the construction is shown in \cite{MR2372849}. In general, let 
\begin{equation*} 
\cfrac{p_{r}}{q_{r}}= 2+\cfrac{1}{2+\cfrac{1}{2+\cdots+\cfrac{1}{2}}},
\end{equation*} 
where the number 2 appears $2r+1$ times. Remember $(B^{3}, S_{r})$ is a rational tangle. There always exists an embedded disk in $B^{3}$ splitting the two strings of the tangle $S_{r}$. The rational number ${p_{r}}/{q_{r}}$ gives rise to a way to find the splitting disk (up to isotopy). 

Precisely, we describe the process of drawing the boundary of the splitting disk in $S^{2}=\partial B^{3}$. 
\begin{enumerate}
	\item Fix the endpoints $A, B, C$ and $D$ of the tangle $S_{r}$ in a great circle of $S^{2}$ as in Figure \ref{ff2}.
	\item Along the great circle, choose $p_{r}$ disjoint points in both the segments $\overrightarrow{AB}$ and $\overrightarrow{CD}$, and choose $q_{r}$ disjoint points in both the segments $\overrightarrow{BC}$ and $\overrightarrow{DA}$. The clockwise orientation is used here.
	\item Take the line through $A$ and $C$ as the axis of symmetry, and connect the chosen points in $\overrightarrow{ABC}$ with the chosen points in $\overrightarrow{ADC}$ by using $p_{r}+q_{r}$ pairwise disjoint simple arcs in the front hemi-sphere of $S^{2}=\partial B^{3}$.
		\item Then take the line through $B$ and $D$ as the axis of symmetry, and connect the chosen points in $\overrightarrow{BAD}$ to the chosen points in $\overrightarrow{BCD}$ by using pairwise disjoint simple arcs in $S^{2}$ just as we did before, but this time, the back hemi-sphere is used.
	\item All the arcs form a simple closed curve in $S^{2}$, denoted $\beta_{r}$. Then $\beta_{r}$ bounds a splitting disk for the tangle $S_{r}$.
	
\end{enumerate}
\begin{figure}
	\centering
		\includegraphics[width=0.8\textwidth]{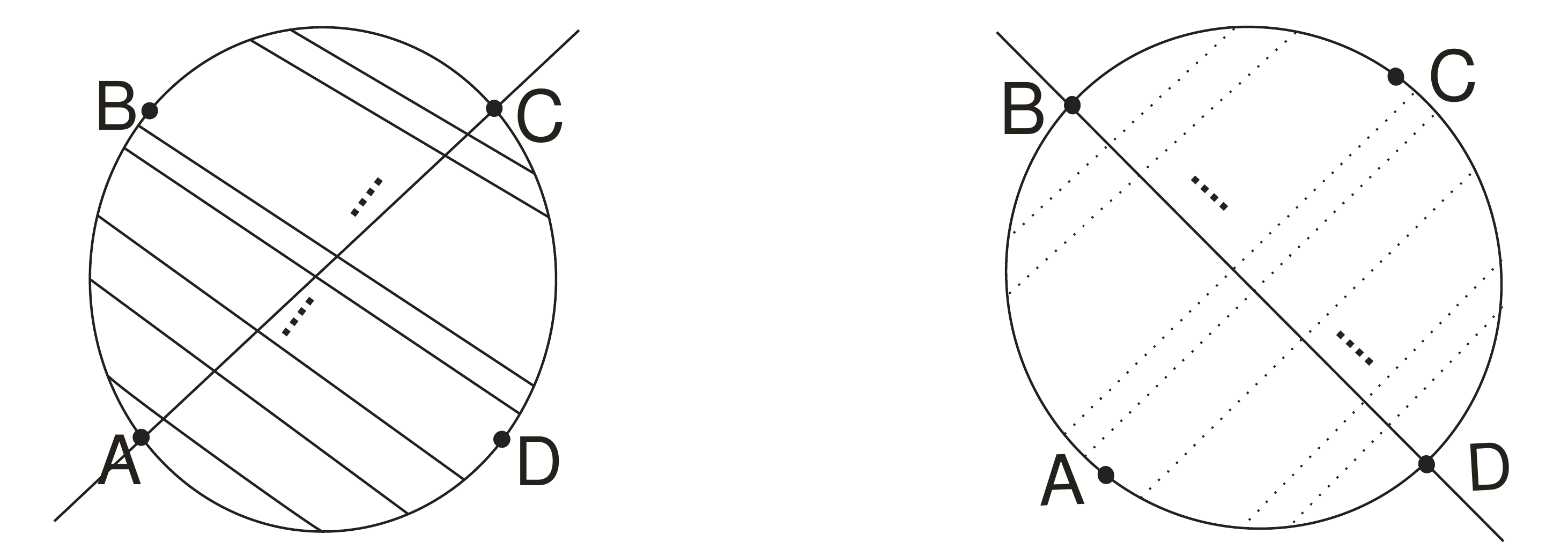}
		\setlength{\abovecaptionskip}{0cm}
	\caption{Connecting chosen points by arcs in the front hemisphere and the back hemisphere.}
	\label{ff2}	
\end{figure}

We attach an unknotted one-handle $h_{AD}$ to $B^{3}$ along the feet $A$ and $D$, and an unknotted one-handle $h_{BC}$ along the feet $B$ and $C$, as shown in Figure \ref{fig:f2}. The resulting manifold is a genus two handlebody, and its boundary is denoted $\Sigma_{2}$. Besides $\beta_{r}$, we define four new curves $\mu, \alpha, \mu_{P_{r}}$ and $\lambda_{P_{r}}$ in $\Sigma_{2}$. The curve $\mu$ is a cocore of the handle $h_{BC}$. The curve $\alpha$ goes along the one-handles $h_{AD}$ and $h_{BC}$ so that attaching a 2-handle to the handlebody along $\alpha$ leaves us a solid torus, which is $\textbf{S}$. The pair $(\mu_{P_{r}}, \lambda_{P_{r}})$ is the preferred meridian-longitude system of $\textbf{S}$, and $\lambda_{P_{r}}\cap \mu_{P_{r}}=\{\theta'\}$. 

We claim that 
$$\operatorname{HD}(S^{3}, P_{r})=(\Sigma_{2}, \{\alpha, \lambda_{P_{r}}\}, \{\beta_{r}, \mu\}, z, w)$$
is a Heegaard diagram for $(S^{3}, P_{r})$. The chain complex defined on it is denoted $\widehat{\operatorname{CFK}}(S^{3}, P_{r})$, and its homology is the knot Floer homology of $C(2r)$. Notice that the intersection $\beta_{r}\cap \lambda_{P_{r}}$ contains $q_{r}$ points, labelled $y_{1}, \cdots, y_{q_{r}}$ from right to left. We can get a Heegaard diagram for $(S^{1}\times S^{2}, P_{r})$ by changing the curve $\lambda_{P_{r}}$ into the curve $\mu_{P_{r}}$ in the Heegaard diagram $\operatorname{HD}(S^{3}, P_{r})$. Precisely, it is
$$\operatorname{HD}(S^{1}\times S^{2}, P_{r})=(\Sigma_{2}, \{\alpha, \mu_{P_{r}}\}, \{\beta_{r}, \mu\}, z, w).$$ The chain complex defined on this Heegaard diagram is denoted $\widehat{\operatorname{CFK}}(S^{1}\times S^{2}, P_{r})$, and its homology is the knot Floer homology of the link $C(2r+1)$ (see \cite{MR2065507} for the knot Floer homology of a link). The intersection $\mu_{P_{r}}\cap\beta_{r}$ contains $2p_{r}$ points, labelled $a_{1}, a_{-1}, \cdots, a_{p_{r}}, a_{-p_{r}}$, as shown in the right-hand figure of Figure \ref{fig:f2}, from bottom to top.

\begin{figure}[h]
	\setlength{\abovecaptionskip}{-0.2cm}
	\centering
		\includegraphics[width=1.0\textwidth]{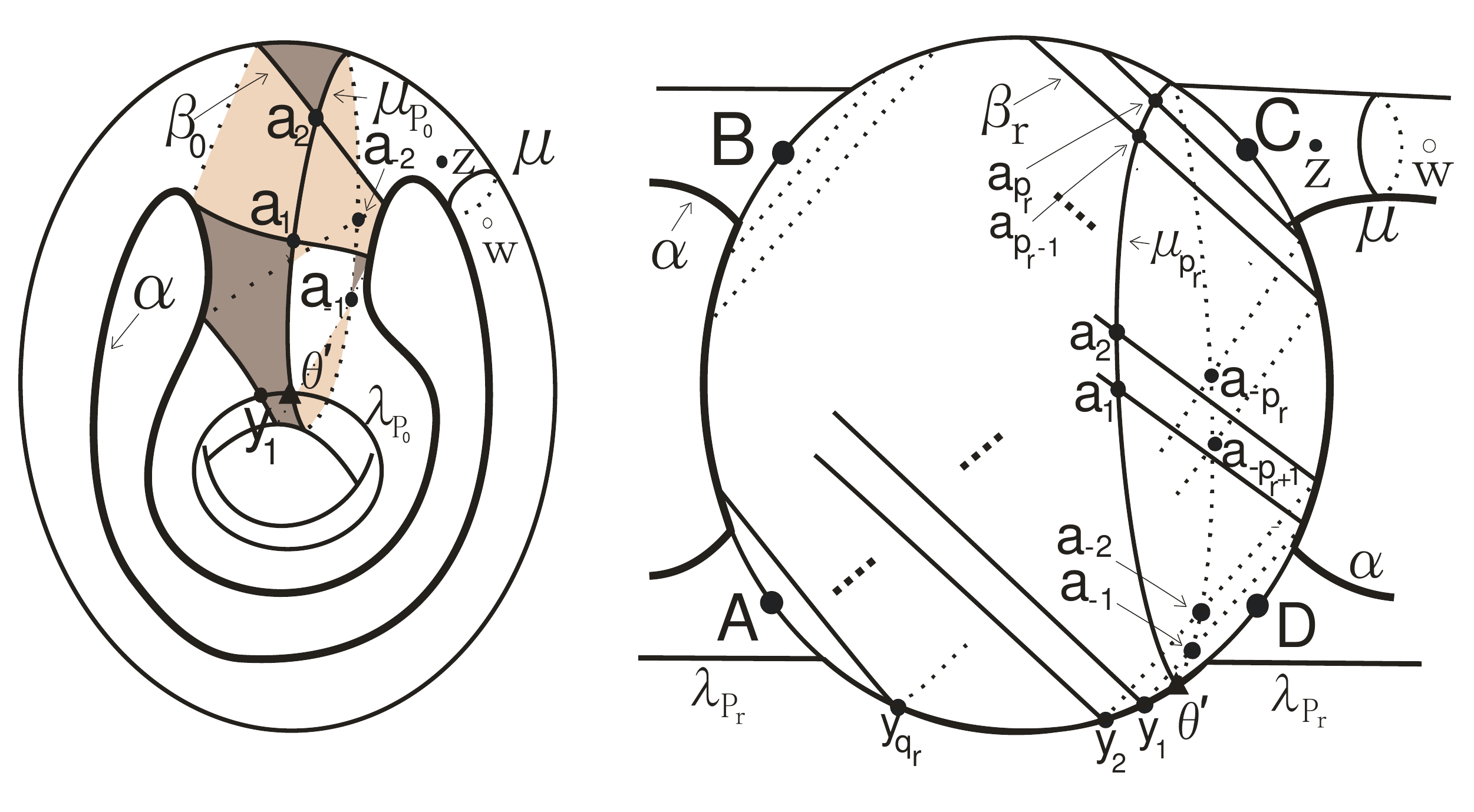}
	\caption{(a) The left-hand figure shows the Heegaard diagrams $\operatorname{HD}(S^{3}, P_{0})$ and $\operatorname{HD}(S^{1}\times S^{2}, P_{0})$, while the right-hand figure is for $\operatorname{HD}(S^{3}, P_{r})$ and $\operatorname{HD}(S^{1}\times S^{2}, P_{r})$ in general case. (b) The shadowed domain on the left-hand side is the generator ${\cal Q}_{0}$ of the periodic domains in $\operatorname{HD}(S^{1}\times S^{2}, P_{0})$.}
	\label{fig:f2}	
\end{figure}

We pause for a while to look at the existence of periodic domains in each diagram. First, we review the definition of a periodic domain. For a Heegaard diagram $(\Sigma_{g}, \boldsymbol{\alpha}, \boldsymbol{\beta}, z, w)$, let $D_{1},\ldots, D_{m}$ denote the closures of the components of $\Sigma_{g}-\boldsymbol{\alpha}-\boldsymbol{\beta}$. 
\begin{defn}
\rm
{\it A periodic domain} is a two-chain ${\cal P}=\sum_{i=1}^{m}a_{i}\cdot D_{i}$ for which the boundary is a sum of $\boldsymbol{\alpha}$ and $\boldsymbol{\beta}$ curves and the local multiplicity $n_{w}({\cal P})$ is zero.
\end{defn}

It is easy to check that the set of periodic domains is a group, and it is isomorphic to $H^{1}(M, {\mathbb Z})$ (refer to \cite{MR2113019}). When the manifold is $S^{3}$, there will be no periodic domains in any Heegaard diagram associated with $S^{3}$ since $H^{1}(S^{3}, {\mathbb Z})=0$. However, the Heegaard diagram $\operatorname{HD}(S^{1}\times S^{2}, P_{r})$ is a diagram for $S^{1}\times S^{2}$, and therefore contains periodic domains since $H^{1}(S^{1}\times S^{2}, {\mathbb Z})={\mathbb Z}$. 

\begin{lemma}
\label{periodic}
Let ${\cal Q}_{r}$ be a generator of the group of periodic domains in $\operatorname{HD}(S^{1}\times S^{2}, P_{r})$. Then
$$n_{z}({\cal Q}_{r})=n_{w}({\cal Q}_{r})=0,$$
for any $r\in {\mathbb Z}_{\geq 0}$.
\end{lemma}
\begin{proof}
We consider the mapping class group of the sphere $S^{2}$ which fixes the endpoints $\{A, B, C, D\}$ pointwisely, and denote it $\operatorname{MCG}(S^{2}, \{A, B, C, D\})$. The lemma is proved by induction on $r$.
When $r=0$, a generator ${\cal Q}_{0}$ for the periodic domains in $\operatorname{HD}(S^{1}\times S^{2}, P_{0})$ is shown in Figure \ref{fig:f2}. We see $\partial{\cal Q}_{0}=\beta_{0}-\mu_{P_{0}}$ and $n_{z}({\cal Q}_{o})=n_{w}({\cal Q}_{0})=0$. 

Assume the lemma holds for some $r\geq 0$. We show that it holds for $r+1$ as well. The change from $(B^{3}, S_{r})$ to $(B^{3}, S_{r+1})$ corresponds to a homeomorphism $f: S^{2} \rightarrow S^{2}$ induced by two full-twists on $S^{2}$. It is easy to see that $f \in \operatorname{MCG}(S^{2}, \{A, B, C, D\})$ and $f(\beta_{r})=\beta_{r+1}$. Let ${\cal Q}_{r+1, 1}=f({\cal Q}_{r})$ and $\mu_{P_{r+1}}'=f(\mu_{P_{r}})$. Notice that there exists an oriented domain bounded by $\mu_{P_{r+1}}'$ and $\mu_{P_{r+1}}$ in $S^{2}-\{A, B, C, D\}$, which we denote ${\cal Q}_{r+1, 2}$. See Figure \ref{fig:f9} for the description. Let ${\cal Q}_{r+1}={\cal Q}_{r+1, 1}+{\cal Q}_{r+1, 2}$. Then ${\cal Q}_{r+1}$ is a generator of the space of periodic domains in $\operatorname{HD}(S^{1}\times S^{2}, P_{r+1})$, and $\partial{\cal Q}_{r+1}=\beta_{r+1}-\mu_{P_{r+1}}$. Since every step happens inside $S^{2}-\{A, B, C, D\}$, it is obvious that $n_{z}({\cal Q}_{r+1})=n_{w}({\cal Q}_{r+1})=0$. The lemma is proved by induction.
\end{proof}

\begin{figure}
	\setlength{\abovecaptionskip}{-0.2cm}
	\centering
		\includegraphics[width=0.8\textwidth]{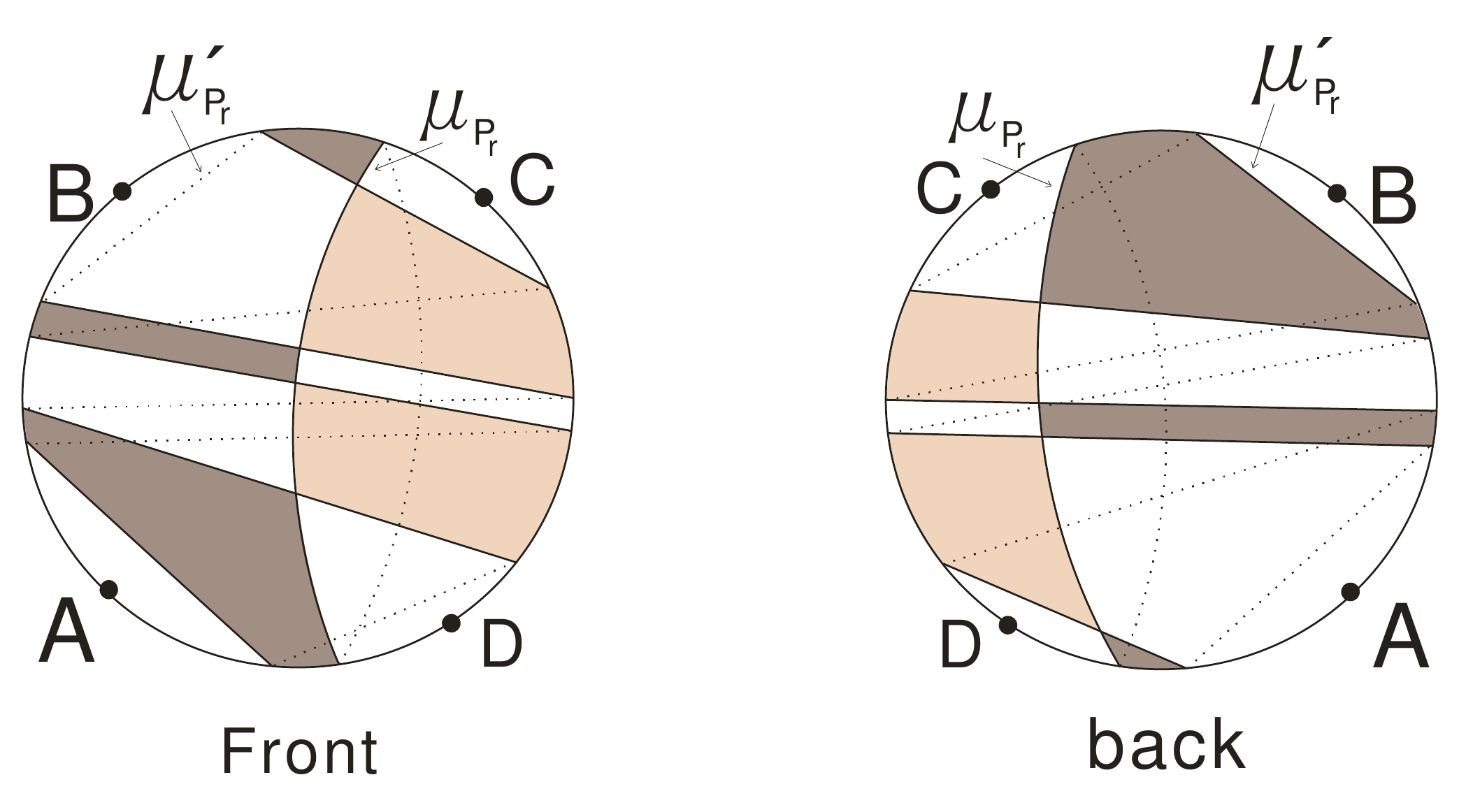}
	\caption{The front and the back sides of the domain ${\cal Q}_{r+1, 2}$ bounded by $\mu_{P_{r}}'$ and $\mu_{P_{r}}$ in $S^{2}$ for any $r\in {\mathbb Z}_{> 0}$.}
	\label{fig:f9}	
\end{figure}

Suppose the Heegaard diagram 
$$\operatorname{HD}(S^{3}, K)=(\Sigma_{g}, \{\alpha_{1}, \alpha_{2}, \cdots, \alpha_{g-1}, \mu_{K}\}, \{\beta_{1}, \beta_{2}, \cdots, \beta_{g}\}, z', w')$$
for the companion knot $K$ is constructed from one of its projections (see \cite{MR1988285} for construction). Here $\mu_{K}$ is a meridian of $K$, and $\mu_{K}\cap (\bigcup_{i=1}^{g}\beta_{i}) = \mu_{K}\cap \beta_{1}= \{x_{0}\}$. See the left-hand picture of Figure \ref{fig:f6} for the diagram around the meridian $\mu_{K}$. Let $\widehat{\operatorname{CFK}}(S^{3}, K)$ be the complex defined on $\operatorname{HD}(S^{3}, K)$. Its homology is the knot Floer homology of $K$. In this Heegaard diagram, one can draw a longitude $\lambda_{K}\subset \Sigma_{g}$ of $K$ with framing $t$ such that $\lambda_{K}\cap (\bigcup_{i=1}^{g-1}\alpha_{i})=\emptyset$. Here the framing of $\lambda_{K}$ is $t$ means that the linking number ${\rm lk}(\lambda_{K}, K)$ is $t$ in $S^{3}$. The longitude $\lambda_{K}$ can be arranged to have one intersection point with $\mu_{K}$. That is, $\lambda_{K}\cap \mu_{K}=\{\theta\}$.

Let $S_{t}^{3}(K)$ be the 3-manifold obtained by $t$-surgery of $S^{3}$ along $K$. Notice that $S_{t}^{3}(K)$ is a rational homology sphere when $t\neq 0$. We assume $t\neq 0$ except in Section 4. If we regard $\mu_{K}$ as a rationally null-homologous simple closed curve in $S_{t}^{3}(K)$, a Heegaard diagram for $(S_{t}^{3}(K), \mu_{K})$ is obtained by replacing $\mu_{K}$ with $\lambda_{K}$ in $\operatorname{HD}(S^{3}, K)$. Specifically, it is
$$\operatorname{HD}(S_{t}^{3}(K), \mu_{K})=(\Sigma_{g}, \{\alpha_{1}, \alpha_{2}, \cdots, \alpha_{g-1}, \lambda_{K}\}, \{\beta_{1}, \beta_{2}, \cdots, \beta_{g}\}, z'', w').$$ Here the placement of the point $z''$ is shown in Figure \ref{fig:f6}. Let $\widehat{\operatorname {CFK}}(S_{t}^{3}(K), \mu_{K})$ denote the complex defined on $\operatorname{HD}(S_{t}^{3}(K), \mu_{K})$. Its homology is $\widehat{\operatorname {HFK}}(S_{t}^{3}(K), \mu_{K})$. See \cite{rational} for the knot Floer homology theory for rationally null-homologous curves. 

\begin{figure}[h]
	\centering
		\includegraphics[width=1.0\textwidth]{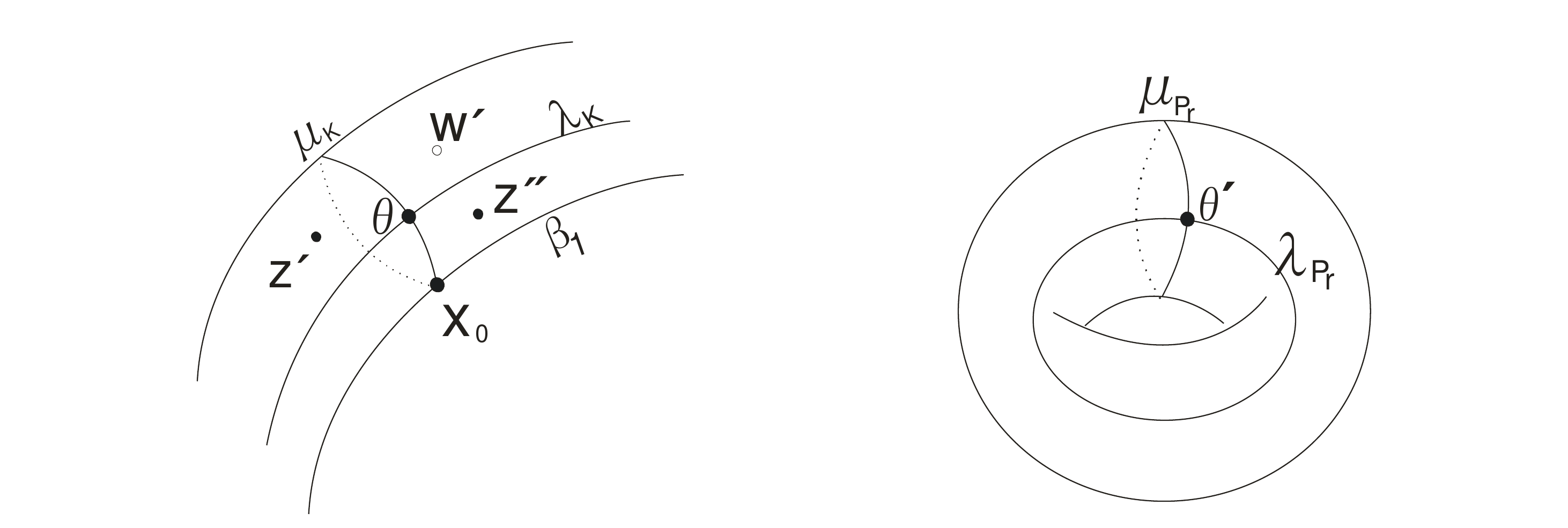}
	\caption{The left-hand figure is a local picture of both $\operatorname{HD}(S^{3}, K)$ and $\operatorname{HD}(S_{t}^{3}(K), \mu_{K})$. The right-hand figure highlights the meridian-longitude system $(\lambda_{P_{r}}, \mu_{P_{r}})$ of the solid torus $\textbf{S}$.}
	\label{fig:f6}	
\end{figure}

As a result, a Heegaard diagram for $K_{t}^{P_{r}}$ is described as follows:
\begin{multline*}
\operatorname{HD}(S^{3}, K_{t}^{P_{r}})=(\Sigma_{g+2}, \boldsymbol{\alpha}=\{\alpha, \alpha_{1}, \alpha_{2}, \cdots, \alpha_{g-1}, \lambda_{P_{r}}\sharp\lambda_{K}, \mu_{P_{r}}\sharp\mu_{K}\},\\
 \boldsymbol{\beta}=\{\beta_{1}, \beta_{2}, \cdots, \beta_{g}, \beta_{r}, \mu\}, z, w).
\end{multline*}
Here $\Sigma_{g+2}$ is the connected sum of $\Sigma_{g}$ and $\Sigma_{2}$ obtained by attaching a one-handle along the feet $\theta$ and $\theta'$. Precisely, we remove the disk neighborhoods of the points $\theta$ and $\theta'$, and attach a one-handle along the boundaries of disks such that the resulting surface is a Heegaard surface of $S^{3}$. The attachment happens inside the tubular neighborhood of $K$. The curve $\lambda_{P_{r}}\sharp\lambda_{K}$ ($\mu_{P_{r}}\sharp\mu_{K}$, respectively) is the band sum of $\lambda_{P_{r}}$ and $\lambda_{K}$ ($\mu_{P_{r}}$ and $\mu_{K}$, respectively) along the one-handle described above (see Figure~\ref{fig:f14} for a detailed illustration). The chain complex defined on $\operatorname{HD}(S^{3}, K_{t}^{P_{r}})$ is denoted $\widehat{\operatorname{CFK}}(S^{3}, K_{t}^{P_{r}})$. Its homology, therefore, is the knot Floer homology of $K_{t}^{P_{r}}$.
 
Attaching a two-handle to the Heegaard surface $\Sigma_{g+2}$ along the band sum of two closed curves is equivalent to identifying these two curves. In the attachment $S^{3}-{\rm Int}N(K_{t}^{P_{r}})=(S^{3}-{\rm Int}N(K))\bigcup_{\partial N(K)=\partial \textbf{S}} (\textbf{S}-{\rm Int}N(P_{r}))$, the meridian $\mu_{K}\subset \partial N(K)$ is identified with the meridian $\mu_{P_{r}}\subset\partial \textbf{S}$, and the longitude $\lambda_{K}\subset \partial N(K)$ is identified with the longitude $\lambda_{P_{r}}\subset\partial \textbf{S}$. Therefore we do the band sums $\lambda_{P_{r}}\sharp\lambda_{K}$ and $\mu_{P_{r}}\sharp\mu_{K}$. One can check that the Heegaard diagram $\operatorname{HD}(S^{3}, K_{t}^{P_{r}})$ satisfies the conditions in Definition \ref{heegaard}.

Here we assume that the orientation of $\Sigma_{g+2}$ is inherited from $\Sigma_{g}$. Explicitly, the local orientation of $\Sigma_{g+2}$ is shown in Figure \ref{fig:f14}.

\begin{figure}
  \setlength{\abovecaptionskip}{-0.2cm}
	\centering
		\includegraphics[width=1.0\textwidth]{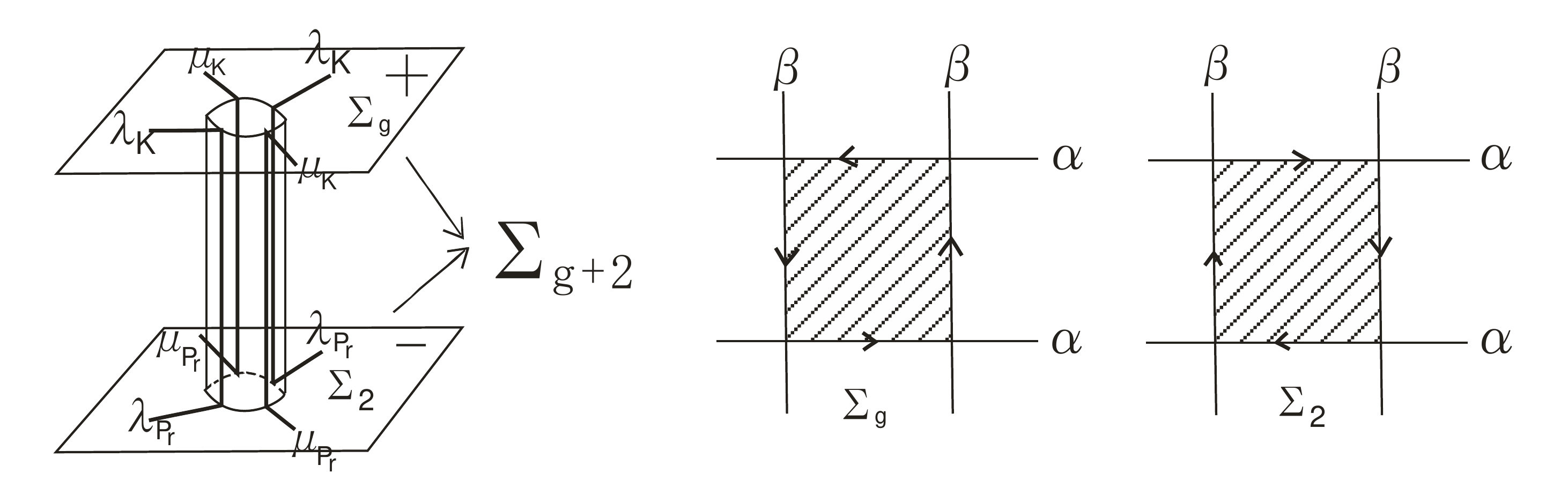}
	\caption{The left-hand figure illustrates the connected sum $\Sigma_{g+2}=\Sigma_{g}\sharp\Sigma_{2}$ along a tube, and the band sums $\lambda_{K}\sharp\lambda_{P_{r}}$ and $\mu_{K}\sharp\mu_{P_{r}}$. The right-hand figure shows the local orientation of each surface.}
	\label{fig:f14}	
\end{figure}

\section{Studying the homology group $\widehat{\operatorname{HFK}}(S^{3}, K_{t}^{P_{r}})$}
In Section 3.1, the Alexander grading of the complex $\widehat{\operatorname{CFK}}(S^{3}, K_{t}^{P_{r}})$ is studied. Then in Section 3.2, restricting to the top Alexander grading, we compute the knot Floer homology $\widehat{\operatorname{HFK}}(S^{3}, K_{t}^{P_{r}})$ for sufficiently large $|t|$, and prove Theorem \ref{main1}. The study here depends heavily on the ideas from Hedden in \cite{MR2372849}.

\subsection{Relation of $\widehat{\operatorname{HFK}}(S^{3}, K_{t}^{P_{r}}, r+1)$ and $\widehat{\operatorname{HFK}}(S_{t}^{3}(K), \mu_{K})$ for $t\neq 0$}
First, the generators of the complex $\widehat{\operatorname {CFK}}(S^{3}, K_{t}^{P_{r}})$ are assorted into two classes. In order to get more accurate classification, the Alexander gradings of some generators are then studied in many aspects. As a result, we find that the top Alexander grading of $\widehat{\operatorname {CFK}}(S^{3}, K_{t}^{P_{r}})$ is $r+1$, which only depends on the pattern $P_{r}$. Moreover, staying at the top grading level, we obtain a parallel result to that of Section 3 in \cite{MR2372849}. That is, there is a natural identification of the chain complexes
$$\widehat{\operatorname {CFK}}(S^{3}, K_{t}^{P_{r}}, r+1)=\bigoplus_{s_{i}\in {\rm Spin}^{c}(S_{t}^{3}(K))}\widehat{\operatorname {CFK}}(S_{t}^{3}(K), \mu_{K}, s_{i}).$$
On the right side, the notation ${\rm Spin}^{c}(S_{t}^{3}(K))$ denotes the set of ${\rm Spin}^{c}$-structures in $S_{t}^{3}(K)$.

Precisely, we claim that the generators of the chain complex $\widehat{\operatorname {CFK}}(S^{3}, K_{t}^{P_{r}})$ are split into two classes of the forms
\begin{enumerate}
	\item $\{x,y_{i}\}\times \textbf{p}\in \widehat{\operatorname {CFK}}(S^{3}, P_{r})\times \widehat{\operatorname {CFK}}(S^{3}, K)$, $i=1,2,\cdots,q_{r}$,
	\item $\{x, a_{j}\}\times \textbf{q}\in \widehat{\operatorname {CFK}}(S^{1}\times S^{2}, P_{r})\times \widehat{\operatorname {CFK}}(S_{t}^{3}(K), \mu_{K})$, $j=\pm1, \pm2, \cdots, \pm p_{r}$.
\end{enumerate}

The claim above is based on the following argument. Recall each generator corresponds to a $(g+2)$-tuple of intersection points between $\boldsymbol{\alpha}$ curves and $\boldsymbol{\beta}$ curves of the Heegaard diagram $\operatorname{HD}(S^{3}, K_{t}^{P_{r}})$. In this Heegaard diagram, we first choose the intersection point for the curves $\mu$ and $\beta_{r}$, which are two $\boldsymbol{\beta}$ curves of $\operatorname{HD}(S^{3}, K_{t}^{P_{r}})$. Notice that the intersection point $x\in\mu\cap\alpha$ has to be chosen since it is the unique choice for $\mu$. For the curve $\beta_{r}$, we can either choose an intersection point in $\{y_{i}\}_{i=1}^{q_{r}}=\beta_{r}\cap\lambda_{P_{r}}\sharp\lambda_{K}$, which constitute the first class, or an intersetion point in $\{a_{j}\}_{j=\pm 1}^{\pm p_{r}}=\beta_{r}\cap\mu_{P_{r}}\sharp\mu_{K}$, which make up the second class (see Figure \ref{fig:f2} for illustration). The intersection points of the other $\boldsymbol{\alpha}$ and $\boldsymbol{\beta}$ curves can be naturally combined to form generators of $\widehat{\operatorname {CFK}}(S^{3}, K)$ or generators of $\widehat{\operatorname {CFK}}(S_{t}^{3}(K), \mu_{K})$.

We now calculate the Alexander grading differences between generators in $\widehat{\operatorname {CFK}}(S^{3}, K_{t}^{P_{r}})$. First let us compare generators with common restriction to their first factors. The proof uses the same ideas as those in the proofs of \cite[Lemmas 3.2 and 3.3]{MR2372849}. 

\begin{lemma}
\label{case1}
\rm
\begin{enumerate}
	 \item {\it Suppose  $\{x, y_{i}\}\times \textbf{p}$ and $\{x, y_{i}\}\times \textbf{p}'$ are two generators of the complex $ \widehat{\operatorname {CFK}}(S^{3}, P_{r})\times \widehat{\operatorname {CFK}}(S^{3}, K)$ for $1\leq i \leq q_{r}$. Then we have
	\begin{equation}
	\label{first}
	A(\{x, y_{i}\}\times \textbf{p})=A(\{x, y_{i}\}\times \textbf{p}').
	\end{equation}}
	\item {\it Suppose  $\{x, a_{j}\}\times \textbf{q}$ and $\{x, a_{j}\}\times \textbf{q}'$ are two generators of the complex $\widehat{\operatorname {CFK}}(S^{1}\times S^{2}, P_{r})\times \widehat{\operatorname {CFK}}(S_{t}^{3}(K), \mu_{K})$ for $j \in \{\pm 1,\cdots, \pm p_{r}\}$. Then we have
	\begin{equation}
	A(\{x, a_{j}\}\times \textbf{q})=A(\{x, a_{j}\}\times \textbf{q}').
	\end{equation}}
\end{enumerate}

\end{lemma}
\begin{proof}
\begin{enumerate}
   \item Suppose $\phi$ is a Whitney disk from $\{x, y_{i}\}\times \textbf{p}$ to $\{x, y_{i}\}\times \textbf{p}'$ with $n_{w}(\phi)=0$. Then the restriction of $\phi$ to $\Sigma_{2}$ must be a periodic domain in $\operatorname{HD}(S^{1}\times S^{2}, P_{r})$, and thus have the form $n\cdot {\cal Q}_{r}$. The first statement of the lemma follows since $n_{z}({\cal Q}_{r})=n_{w}({\cal Q}_{r})=0$ by Lemma \ref{periodic}.
   \item The second statement can be proved by using a similar argument.
	\end{enumerate}
\end{proof}

\begin{figure}
	\centering
		\includegraphics[width=0.7\textwidth]{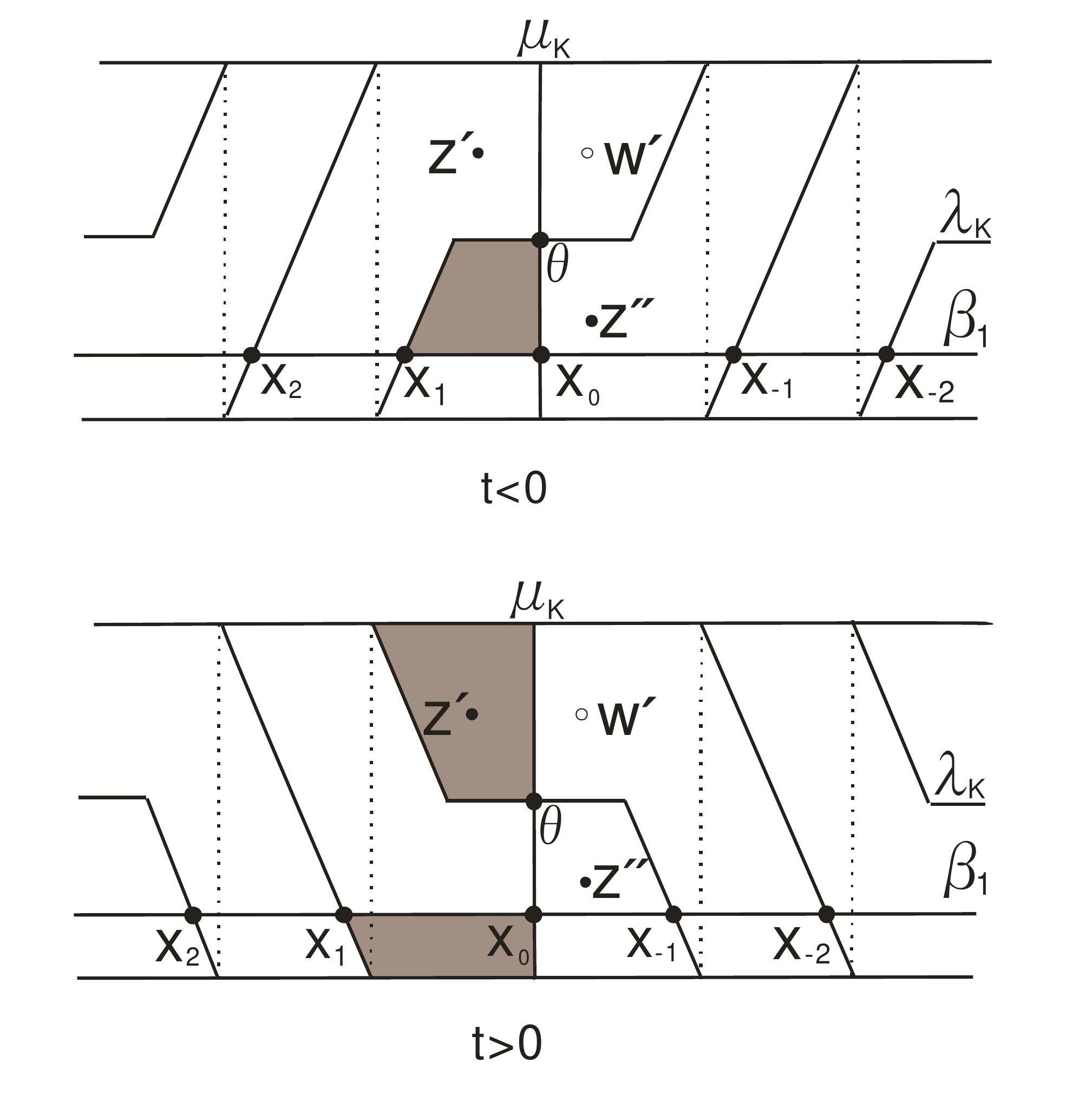}
	\setlength{\abovecaptionskip}{-0.3cm}
	\caption{The triangles $\Delta_{\theta, x_{0}, x_{1}}$ when $t<0$ and $t>0$. }
	\label{fig:f7}	
\end{figure}

\begin{figure}[h]
	\centering
		\includegraphics[width=1.1\textwidth]{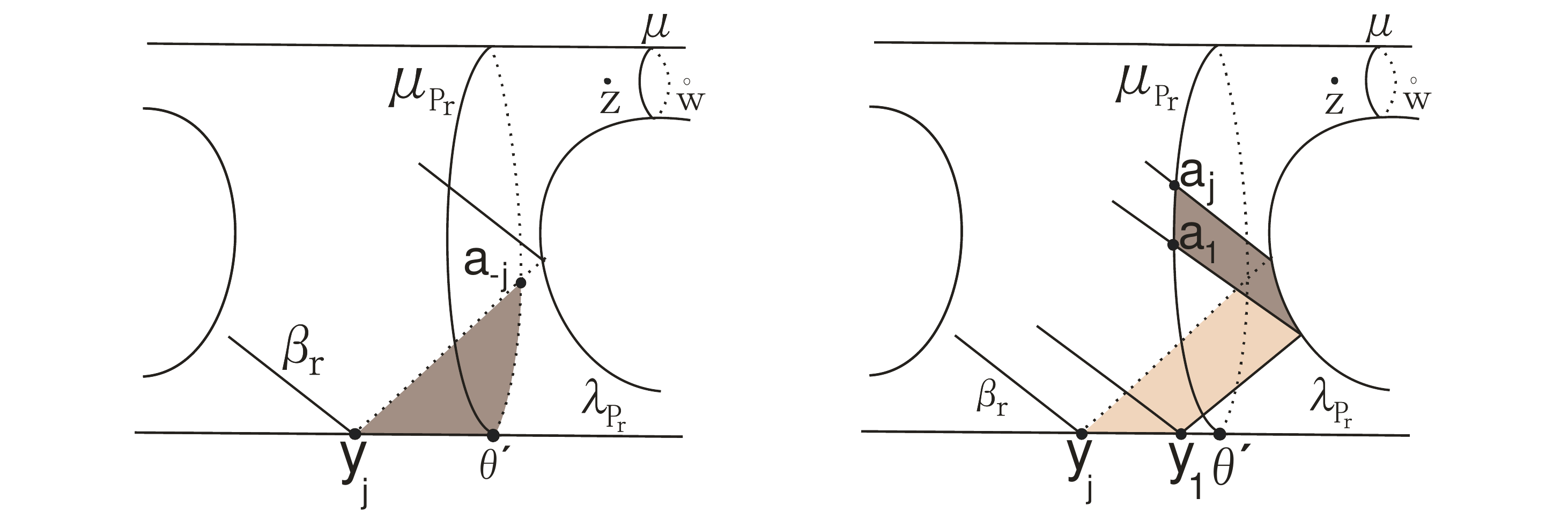}
	\setlength{\abovecaptionskip}{-0.6cm}
	\caption{The left-hand side is the domain $D_{y_{j}, \theta', a_{-j}}$, and the right-hand domain is $D_{y_{j}, a_{j}}$.}
	\label{fig:f8}	
\end{figure}

In the following lemma, we first compare some generators with common restriction to their second factors, and then compare some generators coming from different classes.

\begin{lemma}
\label{where}
\rm
\begin{enumerate}
   \item {\it For each generator $\textbf{q}\in \widehat{\operatorname {CFK}}(S_{t}^{3}(K), \mu_{K})$, we have
   \begin{equation}
   \label{noname}
  A(\{x, a_{j}\}\times \textbf{q})-A(\{x, a_{k}\}\times \textbf{q})=A(\{x, a_{j}\})-A(\{x, a_{k}\}),
  \end{equation}
  for $j,k \in \{\pm 1,\cdots, \pm p_{r}\}$. For any generator $\textbf{p}\in \widehat{\operatorname {CFK}}(S^{3}, K)$, we have
  \begin{equation}
  \label{noname1}
  A(\{x, y_{m}\}\times \textbf{p})-A(\{x, y_{n}\}\times \textbf{p})=A(\{x, y_{m}\})-A(\{x, y_{n}\}),
  \end{equation}
  for $1\leq m, n\leq q_{r}$. 
Here the Alexander gradings on the right sides of {\rm (\ref{noname})} and {\rm (\ref{noname1})} are the gradings in the complexes $\widehat{\operatorname {CFK}}(S^{1}\times S^{2}, P_{r})$ and $\widehat{\operatorname {CFK}}(S^{3}, P_{r})$ respectively.}
	\item {\it For each generator $\textbf{q}\in \widehat{\operatorname {CFK}}(S_{t}^{3}(K), \mu_{K})$, the following hold:
	\begin{equation}
	\label{third}
	\begin{split}
	A(\{x, a_{-j}\}\times \textbf{q})-A(\{x, a_{-2q_{r}+j-1}\}\times \textbf{q})=1,\\
	A(\{x, a_{k}\}\times \textbf{q})-A(\{x, a_{-k}\}\times \textbf{q})=1,	
	\end{split}
	\end{equation}
 for $1\leq j\leq q_{r}$ and $1\leq k\leq p_{r}$.} 
  
   \item {\it There exist a generator $\textbf{p}\in \widehat{\operatorname {CFK}}(S^{3}, K)$ and a generator $\textbf{q}\in \widehat{\operatorname {CFK}}(S_{t}^{3}(K), \mu_{K})$ such that
  \begin{equation}
  \label{fifth}
  A(\{x, y_{j}\}\times \textbf{p})-A(\{x, a_{-j}\}\times \textbf{q})=0,
  \end{equation}
for any $1\leq j \leq q_{r}$.}
\end{enumerate}
\end{lemma}

\begin{proof}
\begin{enumerate}
  	\item Let $\phi$ be a Whitney disk from $\{x, a_{j}\}\times \textbf{q}$ to $\{x, a_{k}\}\times \textbf{q}$ with $n_{w'}(\phi)=0$. Let $\phi|_{\Sigma_{g}}$ denote the restriction of $\phi$ to $\Sigma_{g}$. Then	$$\partial(\phi|_{\Sigma_{g}})=\sum_{i=1}^{g-1}r_{i}\cdot\alpha_{i}+\sum_{j=1}^{g}s_{j}\cdot\beta_{j}+r_{0}\cdot\mu_{K}+s_{0}\cdot\lambda_{K},$$ for some integers $r_{i}, s_{j}$ where $0\leq i\leq g-1$ and $0\leq j\leq g$. First we verify that $s_{0}=0$. This is because the multiplicity of $\lambda_{K}\sharp\lambda_{P_{r}}$ in $\partial \phi$ should be equal to the multiplicity of $\alpha$ in $\partial \phi$, which must be zero. 
  	
Therefore we are able to regard $\phi|_{\Sigma_{g}}$ as a two-chain in the Heegaard diagram $\operatorname{HD}(S^{3}, K)$. We note that its boundary $\partial (\phi|_{\Sigma_{g}})$ consists of $\boldsymbol{\alpha}$ curves and $\boldsymbol{\beta}$ curves in $\operatorname{HD}(S^{3}, K)$. These two facts together with the assumption $n_{w'}(\phi)=0$, that is $n_{w'}(\phi|_{\Sigma_{g}})=0$, imply that $\phi|_{\Sigma_{g}}$ is a periodic domain in $\operatorname{HD}(S^{3}, K)$. Since there is no periodic domain in any Heegaard diagram for $S^{3}$, the domain $\phi|_{\Sigma_{g}}$ must be the empty set. As a result, the domain of $\phi$ is contained in $\Sigma_{2}$, and therefore can be regarded as a Whitney disk from $\{x, a_{j}\}$ to $\{x, a_{k}\}$ in the Heegaard diagram $\operatorname{HD}(S^{1}\times S^{2}, P_{r})$. A similar argument can be used to prove Equation \eqref{noname1}.

	\item The domain shown in the left-hand figure of Figure \ref{fig:f20} connects $\{x, a_{-2q_{r}+j-1}\}$ to $\{x, a_{-j}\}$. The local multiplicities of points $z$ and $w$ are zero and one respectively. Therefore we have
	$$A(\{x, a_{-j}\})-A(\{x, a_{-2q_{r}+j-1}\})=1.$$ By (\ref{noname}), we have
	$$A(\{x, a_{-j}\}\times \textbf{q})-A(\{x, a_{-2q_{r}+j-1}\}\times \textbf{q})=1,$$ for any $1\leq j\leq q_{r}$. On the other hand, for any $1\leq k\leq p_{r}$, there exists a Whitney disk $\phi_{k}$ from $\{x, a_{k}\}\times \textbf{q}$ to $\{x, a_{-k}\}\times \textbf{q}$, as shown in the right-hand figure of Figure \ref{fig:f20} (the shadowed domain). It is easy to calculate that $n_{z}(\phi_{k})=1$ and $n_{w}(\phi_{k})=0$. The result follows from the definition of the Alexander grading.
	\item Let $\textbf{z}$ represent a $(g-1)$-tuple of intersection points between the $\boldsymbol{\alpha}$ curves $\alpha_{1}, \ldots, \alpha_{g-1}$ and the $\boldsymbol{\beta}$ curves $\beta_{2}, \ldots, \beta_{g}$ in $\Sigma_{g}$. Recall that $x_{0}$ is the unique intersection point in $\beta_{1}\cap \mu_{K}$. Let $x_{1}\in \beta_{1}\cap \lambda_{K}$ be the point shown in Figure~\ref{fig:f7}. Then $\textbf{p}=\{x_{0}, \textbf{z}\}$ becomes a generator of $\widehat{\operatorname {CFK}}(S^{3}, K)$ and $\textbf{q}=\{x_{1}, \textbf{z}\}$ is a generator of $\widehat{\operatorname {CFK}}(S_{t}^{3}(K), \mu_{K})$.
	
Let $\Delta_{\theta, x_{0}, x_{1}}$ be the triangle in Figure \ref{fig:f7} which has vertices $x_{1}, x_{0}$ and $\theta$. When $t<0$, a Whitney disk $\phi_{j}$ from $\{x, a_{-j}\}\times \textbf{q}$ to $\{x, y_{j}\}\times \textbf{p}$ can be obtained by making the connected sum of the triangle $\Delta_{\theta, x_{0}, x_{1}}$ and the triangle in $\Sigma_{2}$ which has vertices $y_{j}, a_{-j}$ and $\theta'$ (see the left-hand figure in Figure \ref{fig:f8}). See the left-hand figure of Figure \ref{fig:f16}. The fact $n_{z}(\phi_{j})=n_{w}(\phi_{j})=0$ implies the conclusion.
	
Before discussing the case when $t>0$, we first describe a domain which has vertices $y_{j}, \theta'$ and $a_{j}$, which we denote $D_{y_{j}, \theta', a_{j}}$. Recall in Section 2, the curve $\beta_{r}$ is chosen to be the boundary of the splitting disk of the tangle $(B^{3}, S_{r})$. We can choose the properly embedded disk shown in Figure \ref{fig:f15} as the splitting disk of $(B^{3}, S_{r})$. By local modifications near the endpoints $C$ and $D$ of the tangle $(B^{3}, S_{r})$, this disk is converted into a domain in $\Sigma_{2}$ which has vertices $y_{1}, \theta'$ and $a_{1}$. We call it $D_{y_{1}, \theta', a_{1}}$. Notice that there exists a domain $D_{y_{j}, a_{j}}$ in $\Sigma_{2}$ which has vertices $y_{1}, y_{j}, a_{1}$ and $a_{j}$ for any $1< j\leq q_{r}$, as shown in Figure \ref{fig:f8}. We define $$D_{y_{j}, \theta', a_{j}}:=D_{y_{j}, a_{j}} + D_{y_{1}, \theta', a_{1}}.$$ See the right-hand figure of Figure \ref{fig:f16}.

When $t>0$, let $\phi_{j}$ be the connected sum of the triangle $\Delta_{\theta, x_{0}, x_{1}}$ in $\Sigma_{g}$ and the domain $D_{y_{j}, \theta', a_{j}}$ in $\Sigma_{2}$. See the middle figure of Figure \ref{fig:f16}. Then $\phi_{j}$ is a Whitney disk from $\{x, a_{j}\}\times \textbf{q}$ to $\{x, y_{j}\}\times \textbf{p}$. Since $n_{z}(\phi_{j})=1$ and $n_{w}(\phi_{j})=0$, we get $A(\{x, y_{j}\}\times \textbf{p})-A(\{x, a_{j}\}\times \textbf{q})=-1$. This equation together with (\ref{third}) implies (\ref{fifth}) for the case $t>0$.
\end{enumerate}

\end{proof}

\begin{figure}
  \setlength{\abovecaptionskip}{-0.2cm}
	\centering
		\includegraphics[width=0.9\textwidth]{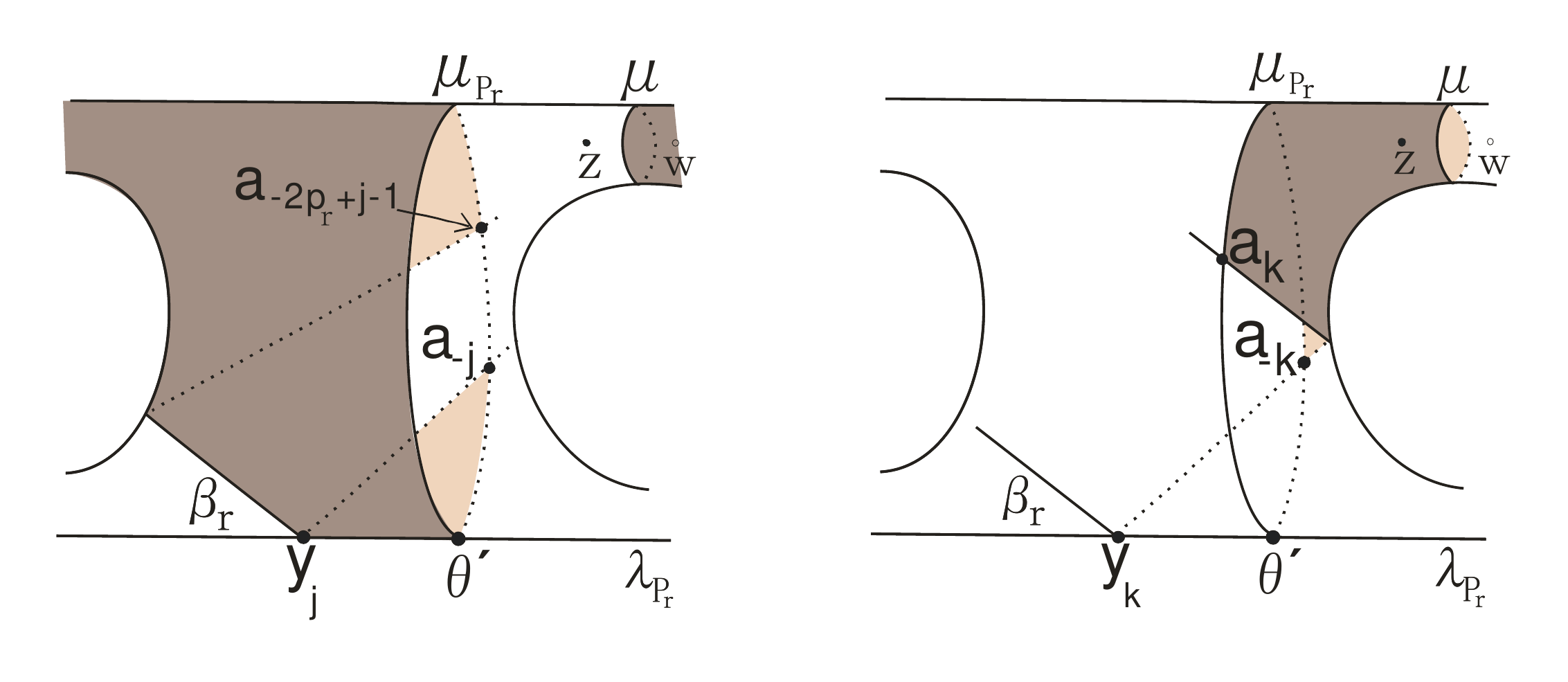}
		\caption{The domain on the left-hand side connects $\{x, a_{-2q_{r}+j-1}\}$ to $\{x, a_{-j}\}$ in $\Sigma_{2}$, while the domain on the right-hand side is from $\{x, a_{k}\}\times \textbf{q}$ to $\{x, a_{-k}\}\times \textbf{q}$ in $\Sigma_{g+2}$.}
	\label{fig:f20}
\end{figure}

\begin{figure}[h]
	\centering
		\includegraphics[width=1.0\textwidth]{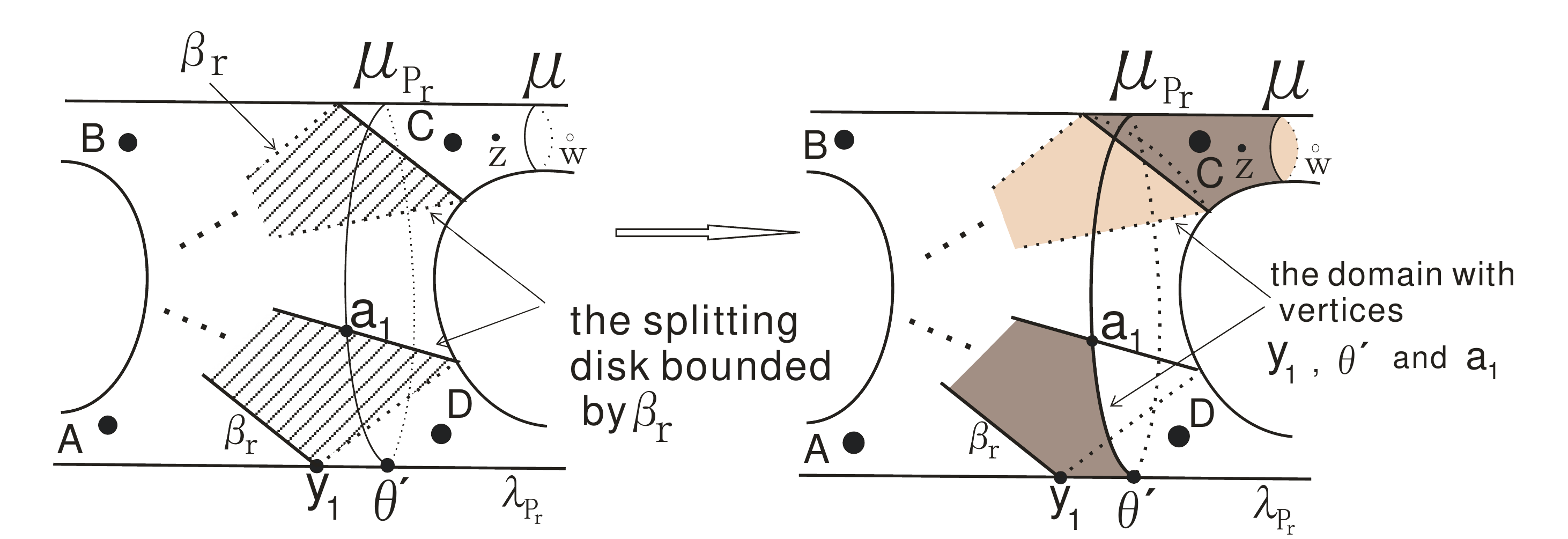}
	\caption{The left-hand figure shows the splitting disk of $(B^{3}, S_{r})$ bounded by $\beta_{r}$. In the right-hand figure the disk is modified to make the domain $D_{y_{1}, \theta', a_{1}}$ in $\Sigma_{2}$.}
	\label{fig:f15}	
\end{figure}

\begin{figure}
  \setlength{\abovecaptionskip}{-0.1cm}
	\centering
		\includegraphics[width=1.0\textwidth]{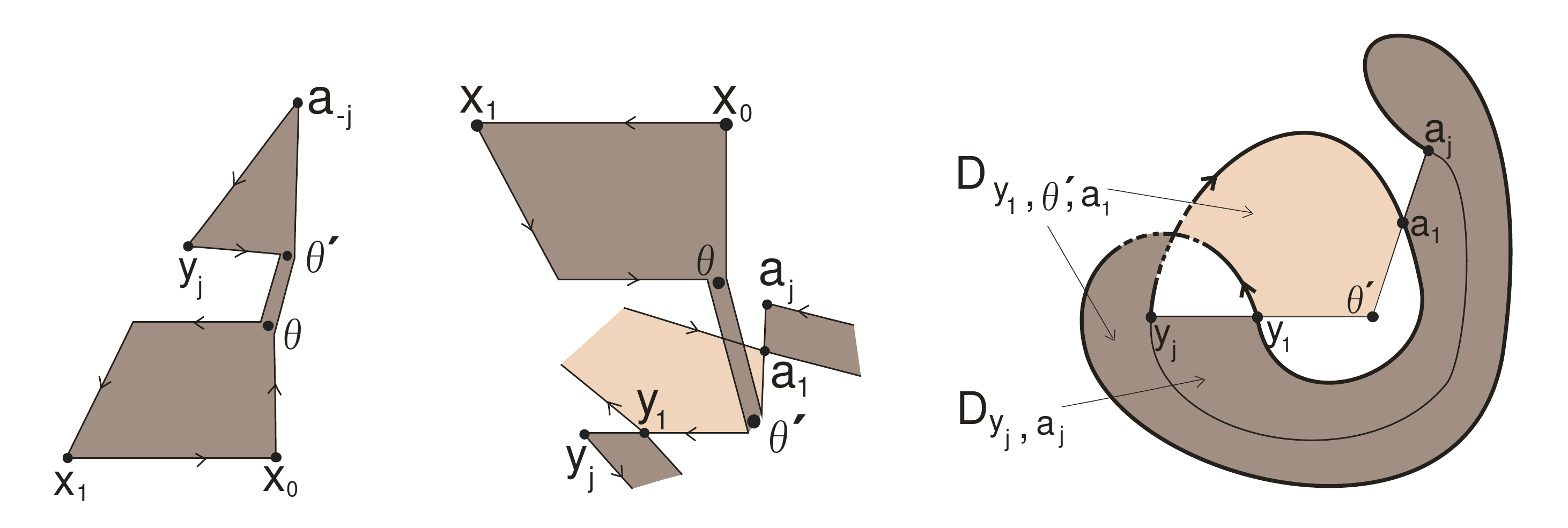}
	\caption{The first two figures show the connected sums of the cases of $t<0$ and $t>0$. The third figure illustrates $D_{y_{j}, \theta', a_{j}}=D_{y_{j}, a_{j}} + D_{y_{1}, \theta', a_{1}}$.}
	\label{fig:f16}	
\end{figure}

Recall that the chain complexes $\widehat{\operatorname {CFK}}(S^{1}\times S^{2}, P_{r})$ and $\widehat{\operatorname {CFK}}(S^{3}, P_{r})$ give rise to the knot Floer homologies of the links $C(2r+1)$ and $C(2r)$ respectively. Let ${\rm det}(L):=\left|\Delta_{L}(-1)\right|$ denote the determinant of a link $L\subset S^{3}$. We state some observations about the determinants of $C(2r+1)$ and $C(2r)$, and about the complexes $\widehat{\operatorname {CFK}}(S^{1}\times S^{2}, P_{r})$ and $\widehat{\operatorname {CFK}}(S^{3}, P_{r})$. 

\begin{lemma}
\begin{enumerate}
\rm
	\item We have {\it ${\rm det}(C(2r+1))=p_{r}$ and ${\rm det}(C(2r))=q_{r}$}. 
	\item {\it The differentials of the complexes $\widehat{\operatorname {CFK}}(S^{1}\times S^{2}, P_{r})$ and $\widehat{\operatorname {CFK}}(S^{3}, P_{r})$ are trivial}.
\end{enumerate}
\end{lemma}
\begin{proof}
\begin{enumerate}
	\item In general, the double-branched cover of $S^{3}$ branched along the two-bridge link $L:=C(b_{1}, b_{2}, \cdots, b_{n})$ is the lens space $L(p, q)$, where 
\begin{equation*} 
\cfrac{p}{q}= b_{1}+\cfrac{1}{b_{2}+\cfrac{1}{b_{3}+\cdots+\cfrac{1}{b_{n}}}}.
\end{equation*} 	
Then
	$$\left|\Delta_{L}(-1)\right|=|H_{1}(L(p, q), {\mathbb Z})|=\left|{\mathbb Z}/{|p|}{\mathbb Z}\right|=|p|.$$
The double-branched covers of $S^{3}$ branched along the two-bridge links $C(2r+1)$ and $C(2r)$ are $L(p_{r}, q_{r})$ and $L(q_{r}, p_{r}-2q_{r})$ respectively. Therefore, we have
	$${\rm det}(C(2r+1))=p_{r} \text{ and } {\rm det}(C(2r))=q_{r}.$$
  \item As recalled in Theorem \ref{euler}, the Euler characteristic of the knot Floer homology of a link is the Alexander-Conway polynomial. In the cases of $C(2r+1)$ and $C(2r)$, we have
  \begin{equation*}
  \begin{split}
  \sum_{i\in {\mathbb Z}}T^{i}\sum_{j\in \frac{1}{2}+{\mathbb Z}}(-1)^{j-1/2}{\rm rk}(\widehat{\operatorname {HFK}}_{j}(S^{3}, C(2r+1), i))& =(T^{1/2}-T^{-1/2})\Delta_{C(2r+1)}(T),\\
  \sum_{i\in {\mathbb Z}}T^{i}\sum_{j\in {\mathbb Z}}(-1)^{j}{\rm rk}(\widehat{\operatorname {HFK}}_{j}(S^{3}, C(2r), i))& =\Delta_{C(2r)}(T).
  \end{split}
  \end{equation*}

When we put $T=-1$, the equalities above simply become:
\begin{equation*}
\begin{split}
\left|\sum_{i\in {\mathbb Z}, j\in \frac{1}{2}+{\mathbb Z}}(-1)^{i+j-1/2}{\rm rk}(\widehat{\operatorname {HFK}}_{j}(S^{3}, C(2r+1), i))\right|& =\left|2\Delta_{C(2r+1)}(-1)\right|=2p_{r},\\
\left|\sum_{i\in {\mathbb Z}, j\in {\mathbb Z}}(-1)^{i+j}{\rm rk}(\widehat{\operatorname {HFK}}_{j}(S^{3}, C(2r), i))\right|& =\left|\Delta_{C(2r)}(-1)\right|=q_{r}.
\end{split}
\end{equation*}

Recall that the numbers of the generators in the complexes $\widehat{\operatorname {CFK}}(S^{1}\times S^{2}, P_{r})$ and $\widehat{\operatorname {CFK}}(S^{3}, P_{r})$ are exactly $2p_{r}$ and $q_{r}$, respectively. Here we are using the integer coefficient for the homologies. Therefore the differentials of the complexes $\widehat{\operatorname {CFK}}(S^{1}\times S^{2}, P_{r})$ and $\widehat{\operatorname {CFK}}(S^{3}, P_{r})$ are trivial.
\end{enumerate}
\end{proof}

On the other hand, since the link $C(2r+1)$ is a fibred alternating link \cite{MR870705}, the polynomial $\Delta_{C(2r+1)}(T)$ is monic with highest power $(r+1)/2$ (refer to \cite{MR2271293}). In other words, the top Alexander grading of $\widehat{\operatorname {HFK}}(S^{3}, C(2r+1))$, which corresponds to the highest power of the polynomial $(T^{1/2}-T^{-1/2})\Delta_{C(2r+1)}(T)$, is $r+1$, and the rank satisfies $${\rm rk}(\widehat{\operatorname {HFK}}(S^{3}, C(2r+1), r+1))=1.$$ That is to say, in the Heegaard diagram $\operatorname{HD}(S^{1}\times S^{2}, P_{r})$, there is a unique intersection point, denoted $a_{f(r)}$, in $\mu_{P_{r}}\cap \beta_{r}$ generating $\widehat{\operatorname {HFK}}(S^{3}, C(2r+1), r+1)$. For convenience, here we define a map $f: {\mathbb Z}_{\geq 0} \longrightarrow {\mathbb Z}_{> 0}$ by sending $r$ to $f(r)$.

From Equations (\ref{first}) to (\ref{fifth}) and the discussion above, we conclude that there are $2r+3$ Alexander gradings in the chain complex $\widehat{\operatorname {CFK}}(S^{3}, K_{t}^{P_{r}})$ at which the complex is non-trivial. The subcomplexes at the top and the bottom gradings are generated by generators of the form $\{x, a_{f(r)}\}\times \widehat{\operatorname {CFK}}(S_{t}^{3}(K), \mu_{K})$ and $\{x, a_{\gamma}\}\times \widehat{\operatorname {CFK}}(S_{t}^{3}(K), \mu_{K})$ for some $-p_{r} \leq \gamma < -q_{r}$, respectively. Therefore, the following identifications exist as groups:
\begin{equation}
\label{group}
\widehat{\operatorname {CFK}}(S^{3}, K_{t}^{P_{r}}, *)\cong \widehat{\operatorname {CFK}}(S_{t}^{3}(K), \mu_{K}),
\end{equation}
where $*$ stands for the top or bottom grading of $\widehat{\operatorname {CFK}}(S^{3}, K_{t}^{P_{r}})$.
As we will see in Proposition \ref{identity}, these identifications can be extended to the isomorphisms of homologies, and the homology groups of $\widehat{\operatorname {CFK}}(S^{3}, K_{t}^{P_{r}})$ at the top and the bottom gradings are non-trivial. Due to the symmetry of knot Floer homology with respect to the Alexander grading, the top and the bottom gradings of the knot Floer homology of $K_{t}^{P_{r}}$ for $t\neq 0$ are $r+1$ and $-r-1$, respectively, and we focus on the homology at the top grading $\widehat{\operatorname {HFK}}(S^{3}, K_{t}^{P_{r}}, r+1)$.

Now we prove Proposition \ref{identity}. A proof when $r=0$ is given in \cite{MR2372849} by Hedden. The general statement here is proved in exactly the same way. For the reader's convenience, we recall the proof.

\begin{prop}
\label{identity}
Let $K$ be a knot in $S^{3}$. Then we have
\begin{equation}
\label{iso}
\widehat{\operatorname {HFK}}(S^{3}, K_{t}^{P_{r}}, *)\cong \bigoplus_{s_{i}\in {\rm Spin}^{c}(S_{t}^{3}(K))}\widehat{\operatorname {HFK}}(S_{t}^{3}(K), \mu_{K}, s_{i}),
\end{equation}
where $*$ stands for the top or bottom grading of $\widehat{\operatorname {CFK}}(S^{3}, K_{t}^{P_{r}})$.
\end{prop}
\begin{proof}
We prove the isomorphism for the top grading case. The other case can be proved by the same argument.
Consider the complex $(\widehat{\operatorname {CFK}}(S^{3}, K_{t}^{P_{r}}, *), \widehat{\partial})$ with $*$ being the top grading. Suppose $\{x, a_{f(r)}\}\times \textbf{p}$ and $\{x, a_{f(r)}\}\times \textbf{q}$ are two generators of the complex, where $\textbf{p}, \textbf{q} \in \widehat{\operatorname {CFK}}(S_{t}^{3}(K), \mu_{K})$. There is a unique Whitney disk $\phi_{\textbf{p}, \textbf{q}}$ in $\Sigma_{g+2}$ with $n_{z}(\phi_{\textbf{p}, \textbf{q}})= n_{w}(\phi_{\textbf{p}, \textbf{q}})=0$ connecting $\{x, a_{f(r)}\}\times \textbf{p}$ to $\{x, a_{f(r)}\}\times \textbf{q}$.

If $\textbf{p}$ and $\textbf{q}$ belong to different elements in ${\rm Spin}^{c}(S_{t}^{3}(K))$, the differential $\widehat{\partial}$ never connects $\{x, a_{f(r)}\}\times \textbf{p}$ to $\{x, a_{f(r)}\}\times \textbf{q}$. This is because the restriction of $\phi_{\textbf{p}, \textbf{q}}$ to $\Sigma_{2}$ is $n\cdot {\cal Q}_{r}$ for some $n\neq 0$. Recall that ${\cal Q}_{r}$ is a generator of the periodic domains in $\operatorname{HD}(S^{1}\times S^{2}, P_{r})$, and it has both positive and negative local multiplicities. Therefore the Whitney disk $\phi_{\textbf{p}, \textbf{q}}$ does not have holomorphic representatives. We can conclude that the differential $\widehat{\partial}$ respects the splitting of $\widehat{\operatorname {CFK}}(S_{t}^{3}(K), \mu_{K})$ along ${\rm Spin}^{c}(S_{t}^{3}(K))$.

Now suppose that $\textbf{p}$ and $\textbf{q}$ belong to the same element in ${\rm Spin}^{c}(S_{t}^{3}(K))$. In this case, the restriction of $\phi_{\textbf{p}, \textbf{q}}$ to $\Sigma_{2}$ is empty. The Whitney disk $\phi_{\textbf{p}, \textbf{q}}$ is completely contained in $\Sigma_{g}$, and therefore can be regarded as a Whitney disk connecting $\textbf{p}$ to $\textbf{q}$ in $\operatorname{HD}(S_{t}^{3}(K), \mu_{K})$, and in addition $n_{w'}(\phi_{\textbf{p}, \textbf{q}})=n_{z''}(\phi_{\textbf{p}, \textbf{q}})= 0$. Therefore $\phi_{\textbf{p}, \textbf{q}}$ contributes equally to the differentials of $\widehat{\operatorname {CFK}}(S^{3}, K_{t}^{P_{r}}, r+1)$ and $\widehat{\operatorname {CFK}}(S_{t}^{3}(K), \mu_{K})$. Conversely given any Whitney disk $\phi'_{\textbf{p}, \textbf{q}}$ in $\operatorname{HD}(S_{t}^{3}(K), \mu_{K})$ connecting $\textbf{p}$ to $\textbf{q}$, if $n_{w'}(\phi'_{\textbf{p}, \textbf{q}})=n_{z''}(\phi'_{\textbf{p}, \textbf{q}})= 0$ it can be naturally regarded as a Whitney disk connecting $\{x, a_{f(r)}\}\times \textbf{p}$ to $\{x, a_{f(r)}\}\times \textbf{q}$ in $\Sigma_{g+2}$, with $n_{w}(\phi'_{\textbf{p}, \textbf{q}})=n_{z}(\phi'_{\textbf{p}, \textbf{q}})= 0$.

In one word, the argument above implies that the differentials on both sides work in the same manner. Therefore, Equation (\ref{iso}) holds.
\end{proof}

\subsection{The group $\widehat{{\rm HFK}}(S^{3}, K_{t}^{P_{r}}, r+1)$ when $|t|\gg 0$}

In this subsection, we adapt the ideas in Section 4 of \cite{MR2372849} to study $\widehat{\operatorname {HFK}}(S^{3}, K_{t}^{P_{r}}, r+1)$ when $|t|$ is sufficiently large. 

As stated in the proof of Lemma \ref{where}, each generator in $\{x, y_{1}\}\times \widehat{\operatorname {CFK}}(S^{3}, K)$ is of the form $\{x, y_{1}\}\times \{x_{0}, \textbf{z}\}$ where $\textbf{z}=(z_{1}, z_{2}, \ldots, z_{g-1})$ for $z_{i}\in \alpha_{i}\cap \beta_{\sigma(i)+1}$ for $1 \leq i\leq g-1$ and $\sigma \in S_{g-1}$, while each generator in $\{x, a_{j}\}\times \widehat{\operatorname {CFK}}(S^{3}_{t}(K), \mu_{K})$ is of the form $\{x, a_{j}\}\times \{x_{\lambda_{K}}, \textbf{z}\}$ for some point $x_{\lambda_{K}}\in\beta_{1} \cap \lambda_{K}\sharp\lambda_{P_{r}}$. The regular neighborhood $\mu_{K}\times I \subset \Sigma_{g}$ of $\mu_{K}$ is called the winding region of the Heegaard diagram $\operatorname{HD}(S^{3}_{t}(K), \mu_{K})$. We assume that for the longitude curve $\lambda_{K}$ of framing $t$, there are $|t|$ intersection points in $\beta_{1} \cap \lambda_{K}\sharp \lambda_{P_{r}}$ supported in the winding region $\mu_{K}\times I$. For convenience, let $\{x_{k}| \left\lfloor {-|t|}/{2}\right\rfloor\leq k \leq \left\lfloor {|t|}/{2}\right\rfloor, k\neq 0\}$ denote the set of these intersection points, where $\left\lfloor u \right\rfloor$ is the integer associated with a number $u$ such that $0 \leq u-\left\lfloor u\right\rfloor<1$. See Figure \ref{fig:f7} for the placement of the points in $\{x_{k}| \left\lfloor {-|t|}/{2}\right\rfloor\leq k \leq \left\lfloor {|t|}/{2}\right\rfloor, k\neq 0\}$. With these assumptions, the intersection point $\{x_{0}, \textbf{z}\}\in \widehat{\operatorname {CFK}}(S^{3}, K)$ is in 1 to $|t|$ correspondence with the intersection points $\{x_{k}, \textbf{z}\}\in \widehat{\operatorname {CFK}}(S_{t}^{3}(K), \mu_{K})$ for $\left\lfloor {-|t|}/{2}\right\rfloor\leq k \leq \left\lfloor {|t|}/{2}\right\rfloor$ ($k\neq 0$) in the winding region.

We recall a theorem from Hedden in \cite{MR2372849}. 

\begin{theo}[Theorem 4.3 in \cite{MR2372849}]
\label{help}
Let $K\subset S^{3}$ be a knot. There exists an integer $ N>0$ such that for $t>N$ the following hold for all $m$.
\begin{align*}
\widehat{\operatorname {HFK}}_{\ast}(S_{t}^{3}(K), \mu_{K}, s_{m})&\cong \widehat{\operatorname {HF}}_{\ast+d_{-}(m)}(F(K,m))\bigoplus \widehat{\operatorname {HF}}_{\ast-2m+d_{-}(m)}(F(K,-m-1)),\\
\widehat{\operatorname {HFK}}_{\ast}(S_{-t}^{3}(K), \mu_{K}, s_{m})&\cong \widehat{\operatorname {HF}}_{\ast-d_{+}(m)}\left(\frac{\widehat{CF}(S^{3})}{F(K,m)}\right)\bigoplus \widehat{\operatorname {HF}}_{\ast-2m-d_{+}(m)}\left(\frac{\widehat{CF}(S^{3})}{F(K,-m-1)}\right),
\end{align*}
where $d_{\pm}(m)=({t-(2m\pm t)^{2}})/{4t}$.
\end{theo}
 
The element $s_{m}$ satisfies a certain condition related to $m$, as stated in \cite{MR2372849} before Theorem 4.3, which will not be recalled here. 

Proposition \ref{identity} and Theorem \ref{help} provide us a relation between $\widehat{\operatorname {HFK}}(S^{3}, K_{t}^{P_{r}}, r+1)$ and the filtered chain homotopy type of $\widehat{\operatorname {CFK}}(S^{3}, K)$, but no information about the Maslov grading is mentioned. In order to solve the problem, we first restate Lemmas~4.5 and 4.6 in \cite{MR2372849}, which study the Maslov gradings of some generators. We modify them into our context as follows.

\begin{lemma}
\label{gradingrelation}
Consider the chain complex $\widehat{\operatorname {CFK}}(S^{3}, K_{t}^{P_{r}})$ for any $r\geq 0$ and $t\neq 0$. Let $\left\lfloor {-|t|}/{2}\right\rfloor\leq k < 0$ and $0 <l \leq  \left\lfloor {|t|}/{2}\right\rfloor$, and $N$ be the integer introduced in Theorem~\ref{help}. For $t>N>0$ we have
\begin{align*}
{\rm gr}(\{x, a_{f(0)}\}\times \{x_{k}, \textbf{z}\})&={\rm gr}(\{x, y_{1}\}\times \{x_{0}, \textbf{z}\})+2A(\{x_{0}, \textbf{z}\})+1,\\
{\rm gr}(\{x, a_{f(0)}\}\times \{x_{l}, \textbf{z}\})&={\rm gr}(\{x, y_{1}\}\times \{x_{0}, \textbf{z}\})+1,
\end{align*}
while for $t<-N<0$, we have
\begin{align*}
{\rm gr}(\{x, a_{f(0)}\}\times \{x_{k}, \textbf{z}\})&={\rm gr}(\{x, y_{1}\}\times \{x_{0}, \textbf{z}\})+2A(\{x_{0}, \textbf{z}\}),\\
{\rm gr}(\{x, a_{f(0)}\}\times \{x_{l}, \textbf{z}\})&={\rm gr}(\{x, y_{1}\}\times \{x_{0}, \textbf{z}\}).
\end{align*}
Here $A(\{x_{0}, \textbf{z}\})$ denotes the Alexander grading of $\{x_{0}, \textbf{z}\}$ in $\widehat{\operatorname {CFK}}(S^{3}, K)$.
\end{lemma}

In the Heegaard diagram $\operatorname{HD}(S^{3}, K_{t}^{P_{i}})$ for $i\geq 1$, there is a Whitney disk $\phi_{i}$ connecting $\{x, a_{f(i-1)}\}\times \textbf{p}$ to some generator $\{x, a_{*}\}\times \textbf{p}$ for any $\textbf{p}\in \widehat{\operatorname {CFK}}(S_{t}^{3}(K), \mu_{K})$. See the shadowed domain in Figure \ref{fig:f17}. We first check that $a_{*}=a_{f(i)}$. Notice that $n_{z}(\phi_{i})=1$ and $n_{w}(\phi_{i})=0$, so we get $$A(\{x, a_{*}\}\times \textbf{p})=A(\{x, a_{f(i-1)}\}\times \textbf{p})+1.$$ The Alexander gradings of the generators $\{x, a_{j}\}\times \textbf{p}$ for $j\in \{\pm 1, \cdots, \pm p_{i-1}\}$ in $\widehat{\operatorname {CFK}}(S^{3}, K_{t}^{P_{i-1}})$ do not change when we regard these generators as in $\widehat{\operatorname {CFK}}(S^{3}, K_{t}^{P_{i}})$. Therefore we have $$A(\{x, a_{*}\}\times \textbf{p})=i+1.$$ The discussion in Section 3.1 tells us $a_{f(i)}$ is the unique point in $\beta_{i}\cap \mu_{P_{i}}$ corresponding to the Alexander grading $i+1$, so $*=f(i)$. 

Consider the Maslov grading.
By Lipshitz's formula (refer to Theorem \ref{lip}), we have $$\mu(\phi_{i})=e(\phi_{i})+n_{\{x, a_{f(i-1)}\}}(\phi_{i})+n_{\{x, a_{f(i)}\}}(\phi_{i}).$$ Here since $\phi_{i}$ can be cut into a hexagon (see Figure \ref{fig:f17}), by Formula (\ref{polygon}) we see $$e(\phi_{i})=1-\frac{3}{2}=-\frac{1}{2}.$$ It is easy to see from Figure \ref{fig:f17} that $$n_{\{x, a_{f(i-1)}\}}(\phi_{i})=\frac{3}{4} \text{\quad and \quad } n_{\{x, a_{f(i)}\}}(\phi_{i})=\frac{3}{4}.$$ Therefore we get $\mu(\phi_{i})=-{1}/{2}+{3}/{4}+{3}/{4}=1$. Since $\mu(\phi_{i})=1$ we get the equation
\begin{equation}
\label{haha}
{\rm gr}(\{x, a_{f(i)}\}\times \{x_{k}, \textbf{z}\})-{\rm gr}(\{x, a_{f(i-1)}\}\times \{x_{k}, \textbf{z}\})=1.
\end{equation}

Notice that Equation (\ref{haha}) does not change when we regard $\{x, a_{f(i)}\}\times \{x_{k}, \textbf{z}\}$ and $\{x, a_{f(i-1)}\}\times \{x_{k}, \textbf{z}\}$ as generators of $\widehat{\operatorname {CFK}}(S^{3}, K_{t}^{P_{r}})$ for any $r> i$. 
Therefore, applying Equation (\ref{haha}) $r$ times gives:
\begin{equation*}
{\rm gr}(\{x, a_{f(r)}\}\times \{x_{k}, \textbf{z}\})-{\rm gr}(\{x, a_{f(0)}\}\times \{x_{k}, \textbf{z}\})=r.
\end{equation*}
This together with Lemma \ref{gradingrelation} implies the following equations. 
Let $\left\lfloor {-|t|}/{2}\right\rfloor\leq k < 0$ and $0 <l \leq  \left\lfloor {|t|}/{2}\right\rfloor$ and $N$ be the integer in Lemma \ref{gradingrelation}. For $t>N>0$, we have
\begin{equation}
\label{relation1} 
\begin{split}
{\rm gr}(\{x, a_{f(r)}\}\times \{x_{k}, \textbf{z}\})&={\rm gr}(\{x, y_{1}\}\times \{x_{0}, \textbf{z}\})+2A(\{x_{0}, \textbf{z}\})+r+1,\\
{\rm gr}(\{x, a_{f(r)}\}\times \{x_{l}, \textbf{z}\})&={\rm gr}(\{x, y_{1}\}\times \{x_{0}, \textbf{z}\})+r+1,
\end{split}
\end{equation}
while for $t<-N<0$, we have
\begin{equation}
\label{relation2}
\begin{split}
{\rm gr}(\{x, a_{f(r)}\}\times \{x_{k}, \textbf{z}\})&={\rm gr}(\{x, y_{1}\}\times \{x_{0}, \textbf{z}\})+2A(\{x_{0}, \textbf{z}\})+r,\\
{\rm gr}(\{x, a_{f(r)}\}\times \{x_{l}, \textbf{z}\})&={\rm gr}(\{x, y_{1}\}\times \{x_{0}, \textbf{z}\})+r. 
\end{split}
\end{equation}

\begin{figure}
	\centering
		\includegraphics[width=1.0\textwidth]{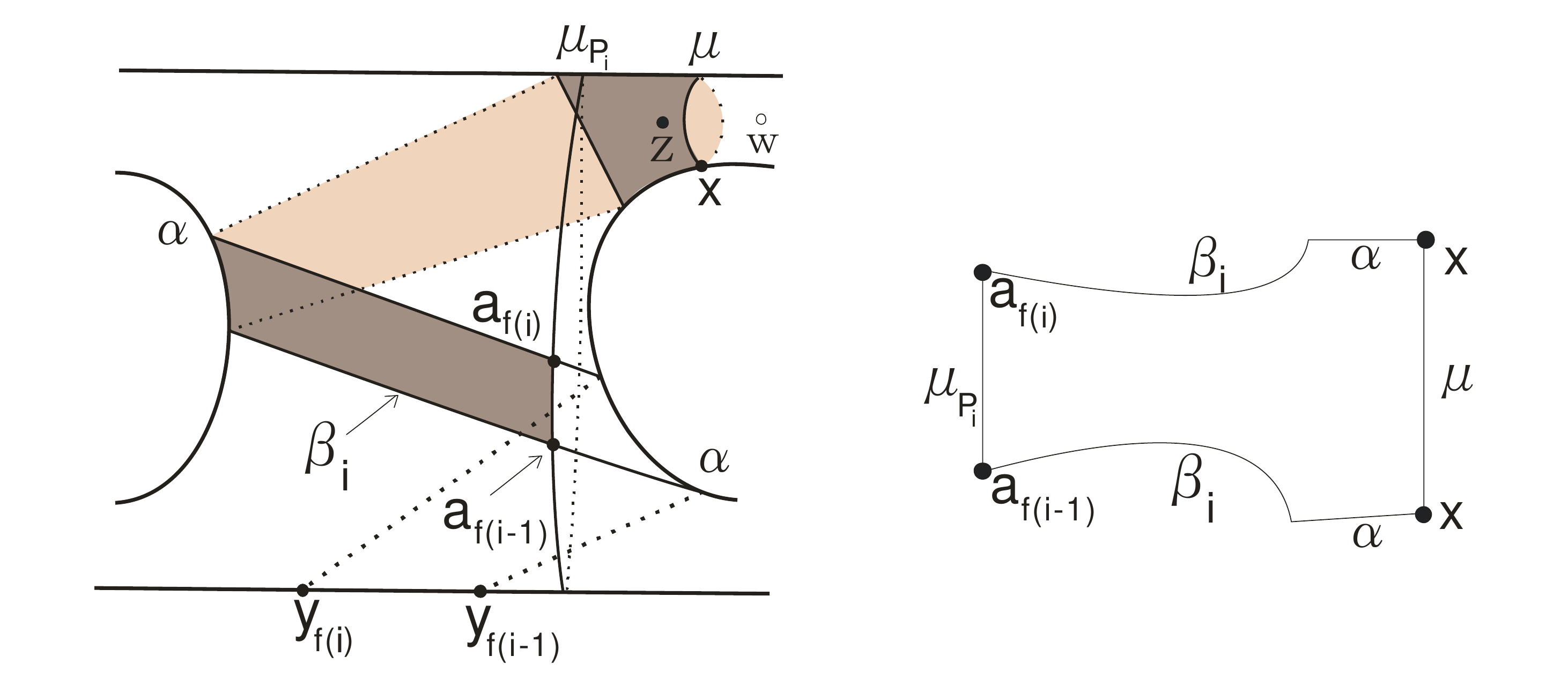}
	\caption{The left-hand figure illustrates the domain $\phi_{i}$ from $a_{f(i)}$ to $a_{f(i-1)}$. The right-hand figure is a hexagon obtained by cutting $\phi_{i}$ along the curve $\alpha$.}
		\label{fig:f17}
\end{figure}

Equations (\ref{relation1}) and (\ref{relation2}) allow us to obtain a similar result to Theorem 4.2 in \cite{MR2372849}. Before stating the result, we prove a lemma.

Recall that any two Heegaard diagrams which specify the same 3-manifold can be connected by a finite sequence of Heegaard moves, which consist of isotopies, handlesides and (de)stablizations. See \cite{MR1501673} for an introduction. In the following lemma, we forget the knots and simply regard $\operatorname{HD}(S^{3}, K_{t}^{P_{r}})$ and $\operatorname{HD}(S^{3}, K)$ as two Heegaard diagrams for $S^{3}$.
\begin{lemma}
\label{moves}
The Heegaard diagram $\operatorname{HD}(S^{3}, K_{t}^{P_{r}})$ can be converted into the Heegaard diagram $\operatorname{HD}(S^{3}, K)$ for $(S^{3}, K)$, by applying a finite sequence of Heegaard moves in the complement of $w$.
\end{lemma}
\begin{proof}
The process converting $\operatorname{HD}(S^{3}, K_{t}^{P_{r}})$ to $\operatorname{HD}'(S^{3}, K)$ is shown in Figure \ref{fig:f19}. Figure~\ref{fig:f19}~(1) shows the Heegaard diagram $\operatorname{HD}(S^{3}, K_{t}^{P_{r}})$, and the curve $\beta_{r}$ bounds a properly embedded disk, say $D_{r}$, in $B^{3}$, which is a splitting disk of the tangle $(B^{3}, S_{r})$. We isotope $D_{r}$ to the position shown in Figure~\ref{fig:f19}~(2) where except the folded region, the disk $D_{r}$ is embedded in $S^{2}=\partial B^{3}$. Then we make a handleslide move of $\beta_{r}$ along $\mu$ as shown in Figure~\ref{fig:f19}~(3). Then isotopy moves change $\beta_{r}$ to the curve in Figure~\ref{fig:f19}~(4). Further isotopy moves occuring in $S^{2}$ change $\beta_{r}$ to the curve shown in Figure~\ref{fig:f19}~(5). The reason we consider the disk $D_{r}$ is to guarantee that the isotopy moves from Figure~\ref{fig:f19}~(4) to Figure~\ref{fig:f19}~(5) occur in the complement of $w$ in $\Sigma_{2}$. The Heegaard moves from Figure~\ref{fig:f19}~(6) to Figure~\ref{fig:f19}~(7) are shown in Figure \ref{fig:f19}. Then we apply two destablizations to the diagram of Figure~\ref{fig:f19}~(7), the first one is between $\mu$ and $\alpha$, and the second one is between $\beta_{r}$ and $\lambda_{P_{r}}$. The final diagram, as shown in Figure~\ref{fig:f19}~(8), is the Heegaard diagram $\operatorname{HD}(S^{3}, K)$. The whole process from $\operatorname{HD}(S^{3}, K_{t}^{P_{r}})$ to $\operatorname{HD}(S^{3}, K)$ is in the complement of $w$.
\end{proof}

\begin{figure}
	\centering
		\includegraphics{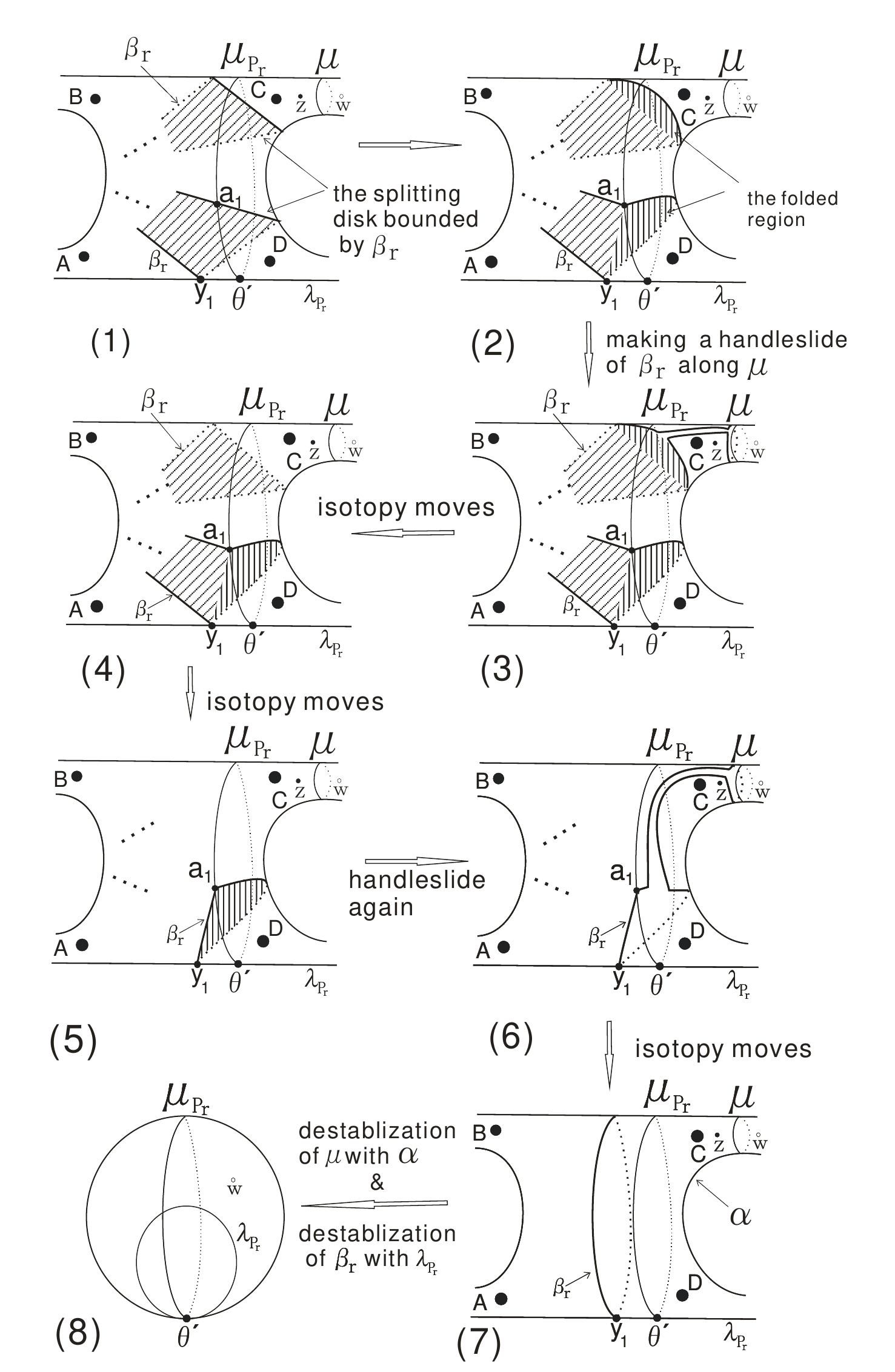}
	\caption{The Heegaard moves from $\operatorname{HD}(S^{3}, K_{t}^{P_{r}})$ to $\operatorname{HD}(S^{3}, K)$.}
	\label{fig:f19}
\end{figure}

Now we show a lemma, which connects $\widehat{\operatorname {HFK}}(S^{3}, K_{t}^{P_{r}}, r+1)$ to the filtered chain homotopy type of $\widehat{\operatorname {CFK}}(S^{3}, K)$.
\begin{lemma}
\label{keytheorem}
For all $t>N>0$, where $N$ is defined as before, there are isomorphisms of ${\mathbb Z}$-graded Abelian groups:
$$\widehat{\operatorname {HFK}}_{\ast}(S^{3}, K_{t}^{P_{r}}, r+1)\cong \bigoplus ^{\left\lfloor {t}/{2}\right\rfloor}_{m=\left\lfloor {-t}/{2}+1\right\rfloor}\left[\widehat{\operatorname {HF}}_{\ast-r-1}(F(K,m))\bigoplus \widehat{\operatorname {HF}}_{\ast-r-1}(F(K,-m-1))\right],$$
$$\widehat{\operatorname {HFK}}_{\ast}(S^{3}, K_{-t}^{P_{r}}, r+1)\cong \bigoplus ^{\left\lfloor {t}/{2}\right\rfloor}_{m=\left\lfloor {-t}/{2}+1\right\rfloor}\left[\widehat{\operatorname {HF}}_{\ast-r}\left(\frac{\widehat{\operatorname {CF}}(S^{3})}{F(K,m)}\right)\bigoplus \widehat{\operatorname {HF}}_{\ast-r}\left(\frac{\widehat{\operatorname {CF}}(S^{3})}{F(K,-m-1)}\right)\right].$$
\end{lemma}
\begin{proof}
The proof goes completely in the same way as that of \cite[Theorem 4.4]{MR2372849}, except for minor differences in the grading shifts on the right-hand side of the equations. As stated in Lemma \ref{moves}, by a sequence of Heegaard moves, each of which happens in the complement of the basepoint $w$, we can convert the Heegaard diagram $\operatorname{HD}(S^{3}, K_{t}^{P_{r}})$ to the Heegaard diagram $\operatorname{HD}(S^{3}, K)$. The Maslov grading of any generator unaffacted by the Heegaard moves, is unchanged throughout the process. It follows that a generator of the form
$$\{x, y_{1}\}\times \textbf{p} \in \{x, y_{1}\}\times\widehat{\operatorname {CFK}}(S^{3}, K)$$ in $\operatorname{HD}(S^{3}, K_{t}^{P_{r}})$ inherits the Maslov grading of $\textbf{p}\in \widehat{\operatorname {CFK}}(S^{3}, K)$ in $\operatorname{HD}(S^{3}, K)$. More precisely we have
\begin{equation}
\label{what}
{\rm gr}(\{x, y_{1}\}\times \textbf{p})= {\rm gr}(\textbf{p}).
\end{equation}

With Equations (\ref{relation1}), (\ref{relation2}) and (\ref{what}), the discussion now follows in the same way as Hedden's proof of \cite[Theorem 4.4]{MR2372849}. The only difference is that we make a shift of $r$ or $r+1$ on the right hand of an equation, where Hedden made a shift of $0$ or $1$. 

\end{proof}

Since $|t|$ can be sufficiently large, we assume that $|t|>2g+4$. Applying the adjunction inequality (refer to \cite{MR2065507}) to Lemma \ref{keytheorem} gets another version of Lemma \ref{keytheorem}:

\begin{lemma}
\label{modi1}
Let $K\subset S^{3}$ be a knot with Seifert genus $g(K)=g$. Then for all $t>N>0$, where $N$ is defined as before, there are isomorphisms of absolutely $\mathbb Z$-graded Abelian groups
\begin{align*}
\widehat{\operatorname {HFK}}_{\ast}(S^{3}, K_{t}^{P_{r}}, r+1)&\cong {\mathbb Z}^{t-2g-2}_{(r+1)}\bigoplus^{g}_{m=-g}\left[\widehat{\operatorname {HF}}_{\ast-r-1}(F(K, m))\right]^{2},\\
\widehat{\operatorname {HFK}}_{\ast}(S^{3}, K_{-t}^{P_{r}}, r+1)&\cong {\mathbb Z}^{t-2g}_{(r)}\bigoplus^{g}_{m=-g}\left[\widehat{\operatorname {HF}}_{\ast-r}\left(\frac{\widehat{\operatorname {CF}}(S^{3})}{F(K, m)}\right)\right]^{2},
\end{align*}
where the subindices $r+1$ and $r$ of ${\mathbb Z}$ denote the Maslov gradings of the terms.
\end{lemma} 

Next, we want to replace $\widehat{\operatorname {HF}}\left({\widehat{\operatorname {CF}}(S^{3})}/{F(K, m)}\right)$ by $\widehat{\operatorname {HF}}(F(K, m))$ in Lemma \ref{modi1}. Here we need a lemma, which has a parallel proof to that of Lemma 5.7 in \cite{MR2372849}.

\begin{lemma}
\label{rankrelation}
Let $K\subset S^{3}$ be a knot with Seifert genus $g(K)=g$, and $t>N>0$ as above. Then
\begin{align*}
\widehat{\operatorname {HFK}}_{\ast}(S^{3}, K_{t}^{P_{r}}, r+1)&=\widehat{\operatorname {HFK}}_{\ast}(S^{3}, K_{-t}^{P_{r}}, r+1) \quad \text{if $\ast\neq r,r+1$,}\\
{\rm rk}(\widehat{\operatorname {HFK}}_{r}(S^{3}, K_{t}^{P_{r}}, r+1))&={\rm rk}(\widehat{\operatorname {HFK}}_{r}(S^{3}, K_{-t}^{P_{r}}, r+1))-t-2\tau(K),\\
{\rm rk}(\widehat{\operatorname {HFK}}_{r+1}(S^{3}, K_{t}^{P_{r}}, r+1))&={\rm rk}(\widehat{\operatorname {HFK}}_{r+1}(S^{3}, K_{-t}^{P_{r}}, r+1))+t-2\tau(K).
\end{align*}
\end{lemma}
\begin{proof}[Proof of Theorem \ref{main1}]
By applying Lemma \ref{rankrelation}, we replace $\widehat{\operatorname {HF}}\left({\widehat{\operatorname {CF}}(S^{3})}/{F(K, m)}\right)$ in the second equation of Lemma \ref{modi1} with $\widehat{\operatorname {HF}}(F(K, m))$. Since the rank of an Abelian group does not count the torsion part, we add the torsion part ``Tor" to complete the proof.
\end{proof}

Theorem \ref{main1} is a rough estimate for $\widehat{\operatorname {HFK}}(S^{3}, K_{t}^{P_{r}}, r+1)$. At present we cannot get more concrete information about $\widehat{\operatorname {HFK}}(S^{3}, K_{t}^{P_{r}})$, as Hedden did for Whitehead doubles in \cite{MR2372849}. However, we claim that Theorem \ref{main1} can be used here to determine the Seifert genus of $K_{t}^{P_{r}}$ for $t=0$. 

\section{The Seifert genus of $K_{t}^{P_{r}}$}

First we remark that there are classical lower bounds to the Seifert genus $g(K)$ of a knot $K\subset S^{3}$. One of them is the degree of the Alexander-Conway polynomial. Precisely it is
\begin{equation}
\label{genus}
g(K)\geq \frac{1}{2}{\rm deg}(\Delta_{K}(T)).
\end{equation}
Recall in the case of $K_{t}^{P_{r}}$, we have Equation \eqref{shape}: $$\Delta_{K_{t}^{P_{r}}}(T)=\Delta_{O_{t}^{P_{r}}}(T).$$ It is easy to deduce that
\begin{equation*}
\frac{1}{2}{\rm deg}(\Delta_{K_{t}^{P_{r}}}(T)) =\frac{1}{2}{\rm deg}(\Delta_{O_{t}^{P_{r}}}(T))=
\begin{cases}
r+1 & \text{if $t\neq 0$},\\
r & \text{if $t=0$}.
\end{cases}
\end{equation*}
Applying Relation (\ref{genus}), we get
\begin{equation*}
g(K_{t}^{P_{r}}) \geq
\begin{cases}
r+1 & \text{if $t\neq 0$},\\
r & \text{if $t=0$}.
\end{cases}
\end{equation*}
On the other hand, a Seifert surface of $K_{t}^{P_{r}}$, which is realized by the surface in Figure~\ref{fig:f4}, is of genus $r+1$. Therefore $r+1$ is an upper bound to $g(K_{t}^{P_{r}})$ for any $t\in {\mathbb Z}$. Therefore
\begin{equation*}
\begin{split}
 g(K_{t}^{P_{r}}) = r+1 \quad& \text{if $t\neq 0$},\\
 r \leq g(K_{t}^{P_{r}}) \leq r+1 \quad & \text{if $t=0$}.
\end{split}
\end{equation*}

Better than Alexander-Conway polynomial, knot Floer homology detects the Seifert genus as stated in Theorem~\ref{seifert}.
We will use Theorems \ref{seifert} and \ref{main1} to prove that $g(K_{t}^{P_{r}})=r+1$ holds for $t=0$. First we prove a lemma on the signature of the link $C(2r+1)$.
\begin{lemma}
The signature of the link $C(2r+1)$ is $-1$.
\end{lemma}
\begin{proof}
Consider the Seifert surface $F_{r}$ in Figure \ref{fig:f4}. A generating set of $H_{1}(F_{r})$ is chosen in Figure \ref{fig:f4}. We have
$$V+V^{t}= 
\left( \begin{array}{ccccc}
-2       & 1 &    &        & \huge{O}  \\
1        & 4 & 1  &        &           \\
         & 1 & -2 & 1      &           \\
         &   &    & \ddots &  1        \\
\huge{O} &   &    & 1      & -2
\end{array} \right)
,$$ where $V$ is the Seifert matrix of $C(2r+1)$ associated with the generating set of $H_{1}(F_{r})$ in Figure \ref{fig:f4}. One can easily prove that the signature of $V+V^{t}$ is $-1$, and so is the signature of $C(2r+1)$.
\end{proof}

\begin{figure}[h]
	\centering
		\includegraphics[width=0.7\textwidth]{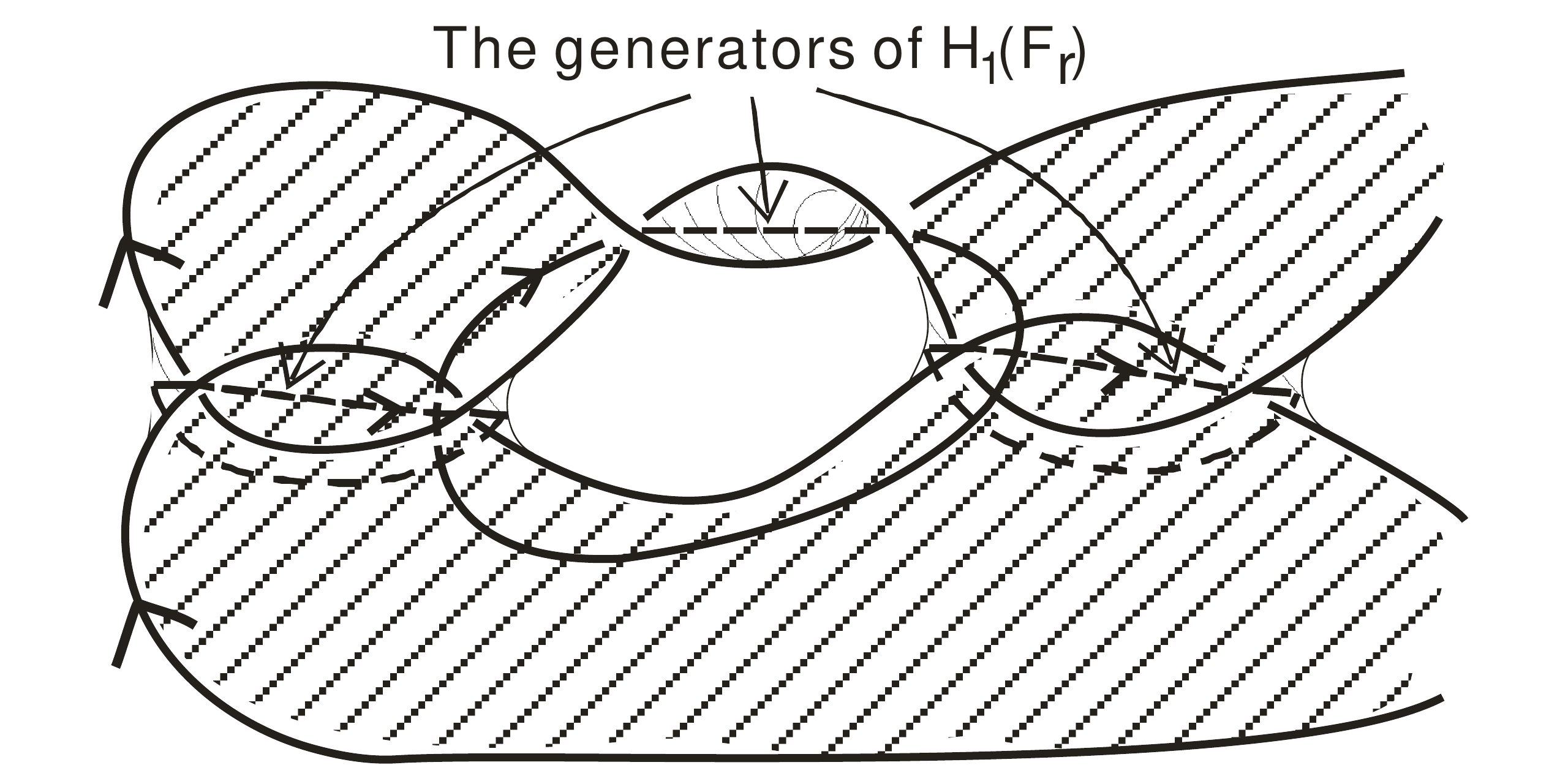}
	\caption{A Seifert surface $F_{r}$ of $C(2r+1)$ and a generating set of $H_{1}(F_{r})$.}
	\label{fig:f4}
\end{figure}

From here, we use ${\mathbb F}:={\mathbb F}_{2}$ instead of ${\mathbb Z}$ as the coefficient of the knot Floer homology. Now we recall a result in \cite{MR2372849}, which will be used later.
\begin{lemma}[Proposition 5.1 in \cite{MR2372849}]
\label{double}
Let $K\subset S^{3}$ be a knot with Seifert genus $g(K)=g$. Then for $t\geq 2\tau(K)$ we have:
$$\widehat{\operatorname {HFK}}_{\ast}(K_{t}^{P_{0}}, 1)\cong {\mathbb F}^{t-2g-2}_{(1)}\bigoplus^{g}_{m=-g}\left[\widehat{\operatorname {HF}}_{\ast-1}(F(K, m))\right]^{2}.$$ 
\end{lemma}

Recall in Section 3.1, we saw that $$\widehat{\operatorname {HFK}}(S^{3}, C(2r+1), r+1)={\mathbb F}_{({(2r+1)}/{2})},$$ where the subindex ${(2r+1)}/{2}$ denotes the Maslov grading. There exists a skein exact sequence connecting $K_{t-1}^{P_{r}}$, $K_{t}^{P_{r}}$ and $C(2r+1)$ as follows (see Section 2.1):
\begin{equation}
\label{sequence}
\cdots\longrightarrow \widehat{\operatorname {HFK}}(S^{3}, K_{t}^{P_{r}}, r+1) \xrightarrow{\,\,f_{1}\,\,} {\mathbb F}_{({(2r+1)}/{2})} \xrightarrow{\,\,f_{2}\,\,} \widehat{\operatorname {HFK}}(S^{3}, K_{t-1}^{P_{r}}, r+1) \xrightarrow{\,\,f_{3}\,\,}\cdots,
\end{equation}
where $f_{1}$ and $f_{2}$ decrease the Maslov grading by ${1}/{2}$, while $f_{3}$ does not increase the grading. In fact, it can be proved that $f_{3}$ preserves the grading. This fact is stated in \cite{MR2372849} for the case $r=0$ . In general, its proof uses the same idea as that of \cite[Proposition 5.8]{MR2372849}.



\begin{lemma}[Proposition 5.8 in \cite{MR2372849}]
\label{preserve}
In the exact sequence {\rm \eqref{sequence}}, $f_{3}$ preserves the Maslov grading.
\end{lemma}

We prove Corollary \ref{main2}.

\begin{proof}[Proof of Corollary \ref{main2}]
First we know $g(K_{0}^{P_{r}})\leq r+1$. We claim that the equality holds by showing that $\widehat{\operatorname {HFK}}(S^{3}, K_{0}^{P_{r}}, r+1)\neq 0$.

If $\widehat{\operatorname {HFK}}(S^{3}, K_{0}^{P_{r}}, r+1)=0$, then
$\widehat{\operatorname {HFK}}_{\ast}(S^{3}, K_{0}^{P_{r}}, r+1)=0, \text{ for $\ast\neq r, r+1$}$. From the exact sequence (\ref{sequence}) combined with Lemma \ref{preserve} we get
\begin{equation*}
\begin{split}
0\longrightarrow \widehat{\operatorname {HFK}}_{*}(S^{3}, K_{s-1}^{P_{r}}, r+1) \xrightarrow{\,\,f_{3}\,\,} \widehat{\operatorname {HFK}}_{*}(S^{3}, K_{s}^{P_{r}}, r+1) \longrightarrow 0, \\
\end{split}
\end{equation*}
for $\ast\neq r, r+1$, which implies that 
\begin{equation*}
\begin{split}
\widehat{\operatorname {HFK}}_{*}(S^{3}, K_{s}^{P_{r}}, r+1)\cong \widehat{\operatorname {HFK}}_{*}(S^{3}, K_{0}^{P_{r}}, r+1)\cong 0,\\
\end{split}
\end{equation*}
for $\ast\neq r, r+1$ and any $s \in {\mathbb Z}$. This property together with Theorem \ref{main1} implies that 
\begin{equation}
\label{property}
\bigoplus^{g}_{m=-g}\left[\widehat{\operatorname {HF}}_{\ast}(F(K, m))\right]^{2}=0, 
\end{equation}
for $\ast\neq -1,0$.

Now we focus on the exact sequence
\begin{multline}
\label{sequence1}
0\longrightarrow \widehat{\operatorname {HFK}}_{r+1}(S^{3}, K_{s-1}^{P_{r}}, r+1)\xrightarrow{\,\,f_{3}\,\,} \widehat{\operatorname {HFK}}_{r+1}(S^{3}, K_{s}^{P_{r}}, r+1) \xrightarrow{\,\,f_{1}\,\,}\\ 
{\mathbb F}_{(\frac{2r+1}{2})}\xrightarrow{\,\,f_{2}\,\,}\widehat{\operatorname {HFK}}_{r}(S^{3}, K_{s-1}^{P}, r+1)
\xrightarrow{\,\,f_{3}\,\,} \widehat{\operatorname {HFK}}_{r}(S^{3}, K_{s}^{P}, r+1)\longrightarrow 0.
\end{multline}

The assumption $\widehat{\operatorname {HFK}}(S^{3}, K_{0}^{P_{r}}, r+1)=0$ implies $\widehat{\operatorname {HFK}}_{*}(S^{3},K_{0}^{P_{r}}, r+1)=0$ for $*=r, r+1$. When $s=1$ by applying the exact sequence (\ref{sequence1}) we get 
$$\widehat{\operatorname {HFK}}_{r}(S^{3}, K_{1}^{P_{r}}, r+1)=0, \widehat{\operatorname {HFK}}_{r+1}(S^{3}, K_{1}^{P_{r}}, r+1)\cong {\mathbb F}.$$
Then when $s=2$, we have
$$\widehat{\operatorname {HFK}}_{r}(S^{3}, K_{2}^{P_{r}}, r+1)=0, \widehat{\operatorname {HFK}}_{r+1}(S^{3}, K_{2}^{P_{r}}, r+1)\cong {\mathbb F}^{2}.$$
We continue to increase $s$ and apply the exact sequence (\ref{sequence1}) iteratively. When $s$ is sufficiently large, say $s=S$ for some large $S>0$, we have 
$$\widehat{\operatorname {HFK}}_{r}(S^{3}, K_{S}^{P_{r}}, r+1)=0, \widehat{\operatorname {HFK}}_{r+1}(S^{3}, K_{S}^{P_{r}}, r+1)\cong {\mathbb F}^{S}.$$

On the other hand, let $s$ start from $0$. Decreasing $s$ and applying (\ref{sequence1}) iteratively, we get
\begin{align*}
s=0: \widehat{\operatorname {HFK}}_{r}(S^{3}, K_{-1}^{P_{r}}, r+1) &\cong{\mathbb F}, \widehat{\operatorname {HFK}}_{r+1}(S^{3}, K_{-1}^{P_{r}}, r+1)=0,\\
s=-1: \widehat{\operatorname {HFK}}_{r}(S^{3}, K_{-2}^{P_{r}}, r+1) &\cong{\mathbb F}^{2}, \widehat{\operatorname {HFK}}_{r+1}(S^{3}, K_{-2}^{P_{r}}, r+1)=0,\\
\vdots\\
s=-S: \widehat{\operatorname {HFK}}_{r}(S^{3}, K_{-S}^{P_{r}}, r+1) &\cong{\mathbb F}^{S}, \widehat{\operatorname {HFK}}_{r+1}(S^{3}, K_{-S}^{P_{r}}, r+1)=0. 
\end{align*}
Let $S>N$, where $N$ is the integer stated in Theorem \ref{main1}. Comparing the knot Floer homologies of $K_{-S}^{P_{r}}$ and $K_{S}^{P_{r}}$ stated above with Theorem \ref{main1}, we get the following restrictions to $\widehat{\operatorname {HF}}(F(K,m))$:
\begin{align*}
\bigoplus^{g}_{m=-g}\left[\widehat{\operatorname {HF}}_{-1}(F(K, m))\right]^{2}&\cong 0,\\
{\mathbb F}^{S-2g-2}\bigoplus^{g}_{m=-g}\left[\widehat{\operatorname {HF}}_{0}(F(K, m))\right]^{2}&\cong {\mathbb F}^{S},\\
{\mathbb F}^{S+2\tau(K)} \bigoplus^{g}_{m=-g}\left[\widehat{\operatorname {HF}}_{-1}(F(K, m))\right]^{2} &\cong {\mathbb F}^{S},\\
{\mathbb F}^{2\tau(K)-2g-2}\bigoplus^{g}_{m=-g}\left[\widehat{\operatorname {HF}}_{0}(F(K, m))\right]^{2}&\cong 0.\\
\end{align*}

Combining the properties above together with (\ref{property}), we can derive the following properites of the knot $K$:
\begin{enumerate}
	\item $\bigoplus^{g}_{m=-g}\left[\widehat{\operatorname {HF}}_{\ast}(F(K, m))\right]^{2}=0 \,\,\,\,\rm{for}\,\, \ast\neq 0$,
	\item $\bigoplus^{g}_{m=-g}\left[\widehat{\operatorname {HF}}_{0}(F(K, m))\right]^{2}\cong {\mathbb F}^{2g+2}$,
	\item $\tau(K)=0$.
\end{enumerate}
If a knot $K\subset S^{3}$ satisfies the properties stated above, Lemma \ref{double} implies that $\widehat{\operatorname {HFK}}(S^{3}, K_{0}^{P_{0}}, 1)=0$, which in turn implies that the genus of $K_{0}^{P_{0}}$ is zero according to Theorem \ref{seifert}. That is to say $K_{0}^{P_{0}}$ is the unknot, which happens only when $K$ itself is the unknot. This contradicts our assumption that $K$ is non-trivial. Therefore, we get $\widehat{\operatorname {HFK}}(K_{t}^{P_{r}}, r+1)\neq 0$, which implies $g(K_{0}^{P_{r}})=r+1$.
\end{proof}

\bibliographystyle{siam}
\bibliography{paper}

\end{document}